\documentclass[11pt,reqno]{amsart}

\newtheorem{proposition}{Proposition}[section]
\newtheorem{lemma}[proposition]{Lemma}
\newtheorem{corollary}[proposition]{Corollary}
\newtheorem{theorem}[proposition]{Theorem}

\theoremstyle{definition}
\newtheorem{definition}[proposition]{Definition}
\newtheorem{example}[proposition]{Example}
\newtheorem{examples}[proposition]{Examples}
\newtheorem{remark}[proposition]{Remark}
\newtheorem{remarks}[proposition]{Remarks}

\newcommand{\thlabel}[1]{\label{th:#1}}
\newcommand{\thref}[1]{Theorem~\ref{th:#1}}
\newcommand{\selabel}[1]{\label{se:#1}}
\newcommand{\seref}[1]{Section~\ref{se:#1}}
\newcommand{\lelabel}[1]{\label{le:#1}}
\newcommand{\leref}[1]{Lemma~\ref{le:#1}}
\newcommand{\prlabel}[1]{\label{pr:#1}}
\newcommand{\prref}[1]{Proposition~\ref{pr:#1}}
\newcommand{\colabel}[1]{\label{co:#1}}
\newcommand{\coref}[1]{Corollary~\ref{co:#1}}
\newcommand{\relabel}[1]{\label{re:#1}}

\newcommand{\exlabel}[1]{\label{ex:#1}}
\newcommand{\exref}[1]{Example~\ref{ex:#1}}
\newcommand{\delabel}[1]{\label{de:#1}}
\newcommand{\deref}[1]{Definition~\ref{de:#1}}
\newcommand{\eqlabel}[1]{\label{eq:#1}}
\newcommand{\equref}[1]{(\ref{eq:#1})}

\newcommand{\Cc}{\mathcal{C}}

\newcommand{\Mm}{\mathcal{M}}

\def\*C{{}^*\hspace*{-1pt}{\Cc}}
\def\text#1{{\rm {\rm #1}}}

\input xy
\xyoption {all} \CompileMatrices

\usepackage{amssymb}
\usepackage{color,amssymb,graphicx,amscd,amsmath}
\usepackage[colorlinks,urlcolor=blue,linkcolor=blue,citecolor=blue]{hyperref}

\begin{document}
\title[The global extension problem for Poisson algebras]
{The global extension problem, crossed products and co-flag
non-commutative Poisson algebras}



\author{A. L. Agore}
\address{Faculty of Engineering, Vrije Universiteit Brussel, Pleinlaan 2, B-1050 Brussels, Belgium}
\address{Permanent address: Department of Applied
Mathematics, Bucharest University of Economic Studies, Piata
Romana 6, RO-010374 Bucharest 1, Romania}
\email{ana.agore@vub.ac.be and ana.agore@gmail.com}

\author{G. Militaru}
\address{Faculty of Mathematics and Computer Science, University of Bucharest, Str.
Academiei 14, RO-010014 Bucharest 1, Romania}
\email{gigel.militaru@fmi.unibuc.ro and gigel.militaru@gmail.com}

\thanks{A.L. Agore is Postdoctoral Fellow of the Fund for Scientific Research Flanders (Belgium) (F.W.O.– Vlaanderen). This work
was supported by a grant of the Romanian National Authority for
Scientific Research, CNCS-UEFISCDI, grant no. 88/05.10.2011.}



\maketitle

\begin{abstract}
Let $P$ be a Poisson algebra, $E$ a vector space and $\pi : E \to
P$ an epimorphism of vector spaces with $V = {\rm Ker} (\pi)$. The
global extension problem asks for the classification of all
Poisson algebra structures that can be defined on $E$ such that
$\pi : E \to P$ becomes a morphism of Poisson algebras. From a
geometrical point of view it means to decompose this groupoid into
connected components and to indicate a point in each such
component. All such Poisson algebra structures on $E$ are
classified by an explicitly constructed classifying set ${\mathcal
G} {\mathcal P} {\mathcal H}^{2} \, (P, \, V)$ which is the
coproduct of all non-abelian cohomological objects ${\mathcal P}
{\mathcal H}^{2} \, (P, \, (V, \cdot_V, [-,-]_V))$ which are the
classifying sets for all extensions of $P$ by $(V, \cdot_V,
[-,-]_V)$. The second classical Poisson cohomology group $H^2 (P,
V)$ appears as the most elementary piece among all components of
${\mathcal G} {\mathcal P} {\mathcal H}^{2} \, (P, \, V)$. Several
examples are provided in the case of metabelian Poisson algebras
or co-flag Poisson algebras over $P$: the latter being Poisson
algebras $Q$ which admit a finite chain of epimorphisms of Poisson
algebras $P_n : = Q \stackrel{\pi_{n}}{\longrightarrow} P_{n-1} \,
\cdots \, P_1 \stackrel{\pi_{1}} {\longrightarrow} P_{0} := P$
such that ${\rm dim} ( {\rm Ker} (\pi_{i}) ) = 1$, for all $i = 1,
\cdots, n$.
\end{abstract}

\section*{Introduction}
A Poisson algebra is both a Lie algebra and an associative algebra
living on the same vector space $P$ such that any bracket $[-, \,
p] : P \to P$ is a derivation of the associative algebra $P$. The
concept is the abstract algebra counterpart of a Poisson manifold:
for a given smooth manifold $M$, there is a one-to-one
correspondence between Poisson brackets on the commutative algebra
$C^{\infty} (M)$ of smooth functions on $M$ and all Poisson
structures on $M$ (see for instance \cite[Remark 1.2]{gra1995}) --
we recall that Poisson structures on $M$ are bivector fields $Q$
such that $[Q, \, Q]$ = 0, where $[-, \, -]$ is the Schouten
bracket of multivector fields. In categorical language the
correspondence $M \mapsto C^{\infty} (M)$ gives a contravariant
functor from the category of Poisson manifolds to the category of
Poisson algebras. The functor $C^{\infty} (-)$ is the tool used
for translating purely geometrical concepts or problems of study
into the algebraic setting of Poisson algebras using a well
established dictionary between the two categories: for details we
refer to \cite{gra2013, LPV} and the references therein. The first
example of a Poisson algebra structure was given by S. D. Poisson
in 1809 related to the study of the three-body problem in
celestial mechanics. Since then, Poisson algebras have become a
very active subject of research in several areas of mathematics
and mathematical physics such as: Hamiltonian mechanics
\cite{arnold, Lie, dirac}, differential geometry \cite{LPV}, Lie
groups and representation theory, noncommutative
algebraic/diferential geometry \cite{bergh}, (super)integrable
systems \cite{marq}, quantum field theory, vertex operator
algebras, quantum groups \cite{chari} and so on. The theory of
Poisson algebras experienced a new upsurge in the last twenty
years in connection to the development of noncommutative geometry.
One of the problems in this context is the right definition of
what should be a \emph{non-commutative} Poisson algebra. From a
purely algebraic point of view a non-commutative Poisson algebra
means a Poisson algebra whose underlying algebra structure is
non-commutative \cite{kobo, kobo2001}, etc. This view point is
also adopted in this work: more precisely, throughout this paper
we do not assume the underlying associative algebra $P$ of a
Poisson algebra to be commutative.

Beyond the remarkable applications in the above mentioned fields,
Poisson algebras are objects of study in their own right, from a
purely algebraic viewpoint \cite{calderon, calderon2, caressa,
casas, Farkas, gozeremm, jordan, kobo, Lu, Lim, umirbaev} etc. In
this spirit a tempting question arises: \emph{for a given positive
integer $n$, classify up to an isomorphism all Poisson algebras of
dimension $n$ over a field $k$.} Having in mind the duality
established by the functor $C^{\infty} (-)$, the question can be
viewed as the algebraic version of its geometric counterpart
initiated in \cite{gra1993} where the first steps towards the
classification of low dimensional Poisson structures is given
using differential geometry tools. The geometrical approach is
also exposed in \cite[Chapter 9]{LPV} where the classification of
Poisson structures of two or three-dimensional manifolds over a
field of characteristic zero is given. The classification of
non-commutative Poisson algebras of a given dimension is a problem
far from being trivial as it contains as subsequent problems the
classification of finite dimensional Lie algebras - known up to
dimension $7$ \cite{popovici} - as well as the classification of
finite dimensional associative algebras - known up to dimension
$5$ \cite{mazzola}. In both cases the classification holds over an
algebraically closed field of characteristic $\neq 2$. The first
steps for the algebraic approach of the classification of Poisson
algebras were taken in \cite[Section 2]{gozeremm}, while
\cite{kobo2001} classifies the finite-dimensional non-commutative
Poisson algebras with Jacobson radical of zero square. The
classification of \emph{finite objects}, such as groups of a given
order or associative (resp. Lie, Poisson, Hopf, etc.) algebras of
a given dimension, relies on the famous \emph{extension problem}
initiated at the level of groups by H\"{o}lder and studied later
on for Lie algebras \cite{CE}, associative algebras \cite{Hoch2},
quantum groups \cite{AD}, Poisson algebras \cite{caressa}, etc.
For the relevance of the extension problem in differential
geometry we refer to \cite{AMR2, le}. A generalization of the
extension problem, called the \emph{global extension problem}, was
introduced in \cite{Mi2013} at the level of Leibniz algebras, as
the categorical dual of what we have called the \emph{extending
structures problem} \cite{am-2011}. The aim of this paper is the
study of the global extension (GE) problem for Poisson algebras
which consists of the following question:

\textbf{The global extension problem.} \emph{Let $P$ be a Poisson
algebra, $E$ a vector space and $\pi : E \to P$ an epimorphism of
vector spaces. Describe and classify all Poisson algebra
structures $(\cdot, [- , - ]) $ that can be defined on $E$ such
that $\pi : E \to P$ is a morphism of Poisson algebras.}

We start by explaining the meaning of the word 'global'. We recall
the classical extension problem formulated in the setting of
Poisson algebras in \cite{caressa}. Let $P$ and $Q$ be two fixed
Poisson algebras. The extension problem asks for the
classification of all Poisson algebras $\mathfrak{E}$ which are
extensions of $P$ by $Q$, i.e. all Poisson algebras $\mathfrak{E}$
that fit into an exact sequence of Poisson algebras:
\begin{eqnarray} \eqlabel{extencros0}
\xymatrix{ 0 \ar[r] & Q \ar[r]^{i} & \mathfrak{E} \ar[r]^{p} & P
\ar[r] & 0 }.
\end{eqnarray}
The classification is up to an isomorphism of Poisson algebras
that stabilizes $Q$ and co-stabilizes $P$. To the best of our
knowledge, the problem was studied so far only for commutative
Poisson algebras and moreover $Q$ was considered to be the abelian
Poisson algebra (i.e. the multiplication and the bracket on $Q$
are both the trivial maps). In this case the isomorphism classes
of all extensions of $P$ by $Q$ are parameterized by the second
Poisson cohomology group $H^2 (P, Q)$ \cite[Proposition
6.1]{caressa}. Now, if $(E, \cdot_E, [-, -]_E)$ is a Poisson
algebra structure on $E$ such that $\pi: (E, \cdot_E, [-, -]_E)
\to P$ is a morphism of Poisson algebras, then $(E, \cdot_E, [-,
-]_E)$ is an extension of $P$ by $V: = {\rm Ker} (\pi)$, which is
a Poisson subalgebra of $(E, \cdot_E, [-, -]_E)$. However, this
Poisson algebra structure of $V$ is not fixed from the input data
as in the case of the extension problem: it depends essentially on
the Poisson algebra structures on $E$ which we are looking for.
Thus, we can conclude that the classical extension problem is the
\emph{'local'} version of the GE-problem: namely, the case when
the Poisson algebra structures on ${\rm Ker} (\pi)$ are fixed.
Secondly, we will explain what we mean by classification in the
GE-problem and its geometrical interpretation. We denote by
${\mathcal C}_{\pi} \, (E, \, P)$ the small category whose objects
are all Poisson algebra structures $(\cdot_E, [- , - ]_E)$ on $E$
such that $\pi: (E, \cdot_E, [- , - ]_E) \to P$ is a Poisson
algebra map. A morphism $\varphi: (\cdot_E, [- , - ]_E) \to
(\cdot'_E, [- , - ]'_E)$ in the category ${\mathcal C}_{\pi} \,
(E, \, P)$ is a Poisson algebra map $\varphi: (E, \cdot_E, \{-, \,
-\}_E) \to (E, \cdot'_E, \{-, \, -\}^{'}_E)$ which stabilizes $V$
and co-stabilizes $P$, i.e. the following diagram
\begin{eqnarray} \eqlabel{diagramaintro}
\xymatrix {& V \ar[r]^{i} \ar[d]_{Id} & {E}
\ar[r]^{\pi} \ar[d]^{\varphi} & P \ar[d]^{Id}\\
& V \ar[r]^{i} & {E}\ar[r]^{\pi } & P}
\end{eqnarray}
is commutative. In this case we say that the Poisson algebra
structures $(\cdot_E, [- , - ]_E)$ and $(\cdot'_E, [- , - ]'_E)$
on $E$ are \emph{cohomologous} and we denote this by $(E, \cdot_E,
\{-, \, -\}_E) \approx (E, \cdot'_E, \{-, \, -\}^{'}_E)$. The
category ${\mathcal C}_{\pi} \, (E, \, P)$ is a groupoid, i.e. any
morphism is an isomorphism. In particular, we obtain that
$\approx$ is an equivalence relation on the set of objects of
${\mathcal C}_{\pi} \, (E, \, P)$ and we denote by ${\rm Gext} \,
(P, \, E)$ the set of all equivalence classes via $\approx$, i.e.
${\rm Gext} \, (P, \, E) := {\mathcal C}_{\pi} (E, \, P)/\approx$.
The answer to the GE-problem will be provided by explicitly
computing ${\rm Gext} \, (P, \, E)$ for a given Poisson algebra
$P$ and a vector space $E$. From geometrical view point this means
to give the decomposition of the groupoid ${\mathcal C}_{\pi} \,
(E, \, P)$ into connected components and to indicate a 'point' in
each such component. This will be the main result of the paper: we
shall prove that ${\rm Gext} (P, \, E)$ is parameterized by a
global cohomological object denoted by ${\mathcal G} {\mathcal P}
{\mathcal H}^{2} \, (P, \, V)$, where $V = {\rm Ker}(\pi)$. The
explicit bijection between ${\mathcal G} {\mathcal P} {\mathcal
H}^{2} \, (P, \, V)$ and ${\rm Gext} (P, \, E)$ is also indicated.
Moreover, we shall prove that ${\mathcal G} {\mathcal P} {\mathcal
H}^{2} \, (P, \, V)$ is the coproduct in the category of sets of
all classifying objects of all 'local' extension problems, which
are the \emph{non-abelian} cohomological objects denoted by
${\mathcal P} {\mathcal H}^{2} \, (P, \, (V, \cdot_V, [-, -]_V ))$
-- for any Poisson algebra structure $(\cdot_V, [-, -]_V )$ on
$V$. The second classical Poisson cohomology group $H^2 (P, V)$
\cite{caressa, Flato} appears as the most elementary piece among
all components of ${\mathcal G} {\mathcal P} {\mathcal H}^{2} \,
(P, \, V)$. The results are proved in the general case, leaving
aside the 'abelian' case which is the traditional setting for the
extension problem for groups or Lie (resp. associative, Poisson)
algebras.

The paper is organized as follows: in \seref{prel} we recall the
basic concepts that will be used throughout the article following
the terminology of \cite{am-2013, am-2013c}. \seref{unifiedprod}
contains the main results of the paper: the crossed product of
Poisson algebras is introduced as the key tool of our approach.
The crossed product, denoted by $P \, \sharp \, V$ is associated
to two Poisson algebras $P$ and $V$ connected by three actions
$(\rightharpoonup, \, \triangleleft, \, \triangleright)$ and two
non-abelian cocycles $\vartheta : P \times P \to V$ and $f: P
\times P \to V$. The datum $(\rightharpoonup, \, \triangleleft, \,
\triangleright, \, \vartheta, \, f)$ must satisfy several axioms
in order to obtain a crossed product of Poisson algebras as given
in \thref{1}. Let $\pi : E \to P$ be an epimorphism of vector
spaces between a vector space $E$ and a Poisson algebra $P$ with
$V = {\rm Ker} (\pi)$. We prove that any Poisson algebra structure
$(\cdot_E, [-, \, -]_E)$ on $E$ such that $\pi : (E, \cdot_E, [-,
\, -]_E) \to P$ is a morphism of Poisson algebras is cohomologous
to a crossed product $P \, \sharp \, V$. Based on this, the
theoretical answer to the GE-problem is given in \thref{main1222}:
the classifying set ${\rm Gext} (P, \, E)$ is parameterized by an
explicitly constructed global cohomological object ${\mathcal G}
{\mathcal P} {\mathcal H}^{2} \, (P, \, V)$ and the bijection
between the elements of ${\mathcal G} {\mathcal P} {\mathcal
H}^{2} \, (P, \, V)$ and ${\rm Gext} (P, \, E)$ is given. The
relation 'global' vs 'local' for the extension problem is given in
\coref{formulamare}. For a fixed Poisson algebra $V = (V, \cdot_V,
[-, \, -]_V)$ the classifying object ${\mathcal P} {\mathcal
H}^{2} \, (P, \, (V, \cdot_V, [-, \, -]_V))$ of all extensions of
$P$ by $V$ is constructed as an answer to the classical extension
problem for Poisson algebras leaving aside the abelian case. Then
we show that ${\mathcal G} {\mathcal P} {\mathcal H}^{2} \, (P, \,
V)$ is the coproduct in the category of sets of all local
non-abelian cohomological groups ${\mathcal P} {\mathcal H}^{2} \,
(P, \, (V, \cdot_V, [-, \, -]_V))$, the coproduct being made over
all possible Poisson algebra structures $(\cdot_V, \{-, -\}_V )$
on $V$. The abelian case, corresponding to the trivial Poisson
structure on $V$ is derived as a special case. The concept of
metabelian Poisson algebra is introduced and their structure is
explicitly described. In \seref{coflag} we shall identify a way of
computing ${\mathcal G} {\mathcal P} {\mathcal H}^{2} \, (P, \,
V)$ for what we have called co-flag Poisson algebras over $P$ as
defined in \deref{coflg}. All co-flag Poisson structures over $P$
can be completely described by a recursive reasoning, the key step
being treated in \thref{clsfP1}. For instance, \exref{calexpext}
describes and classifies the $4$-dimensional co-flag algebras over
the Heisenberg Poisson algebra. At the end of the paper, as an
application of our constructions, we take the first steps toward
the classification of co-flag Poisson algebras.
\coref{dim2Poissoncoflag} provides the description and
classification of all $2$-dimensional co-flag Poisson algebras.
Using the recursive method introduced in \seref{coflag} we can
then describe all $3$-dimensional co-flag Poisson algebras: a
relevant example is given in \coref{dim2Poissoncoflag}.

\section{Preliminaries}\selabel{prel}
For two sets $X$ and $Y$ we shall denote by $X \sqcup Y$ their
coproduct in the category of sets, i.e. $X \sqcup Y$ is the
disjoint union of $X$ and $Y$.  All vector spaces, Lie or
associative algebras, linear or bilinear maps are over an
arbitrary field $k$. A map $f: V \to W$ between two vector spaces
is called the \emph{trivial} map if $f (v) = 0$, for all $v\in V$.
Let $\pi : E \to P$ be a linear map between two vector spaces and
$V: = {\rm Ker} (\pi)$. We say that a linear map $\varphi: E \to
E$ \emph{stabilizes} $V$ (resp. \emph{co-stabilizes} $P$) if the
left square (resp. the right square) of diagram
\equref{diagramaintro} is commutative. Throughout this paper, by
an algebra $P = (P, \, m_P)$ we will always mean an associative,
not necessarily commutative or unital, algebra over $k$. The
bilinear multiplication $m_P: P \times P \to P$ will be denoted by
$m_P (p, \, q) = pq$, for all $p$, $q\in P$. When $P = (P, \,
m_P)$ is unital the unit will be denoted by $1_P$. The following
convention will be used: when defining the Poisson algebra
structure we only write down the non-zero values of the
multiplication and the bracket. For an algebra $P$ we shall denote
by ${}_P\Mm_P$ the category of all $P$-bimodules, i.e. triples
$(V, \, \rightharpoonup, \, \triangleleft)$ consisting of a vector
space $V$ and two bilinear maps $\rightharpoonup \, : P \times V
\to V$, $\triangleleft : V \times P \to V$ such that $(V,
\rightharpoonup)$ is a left $P$-module, $(V, \triangleleft)$ is a
right $P$-module and $p \rightharpoonup (x \triangleleft q) = (p
\rightharpoonup x) \triangleleft q$, for all $p$, $q\in P$ and
$x\in V$. Representations of a Lie algebra $P = (P, \, [-, \, -])$
will be viewed as left modules over $P$. Explicitly, a left Lie
$P$-module is a vector space $V$ together with a bilinear map $
\triangleright : P \times V \to V$ such that for any $p$, $q \in
P$ and $x\in V$ we have:
\begin{equation}\eqlabel{modulstring}
[p, \, q] \triangleright x = p \triangleright (q \triangleright x)
- q \triangleright (p \triangleright x).
\end{equation}
The category of left Lie $P$-modules will be denoted by ${}^P\Mm$.
A \emph{Poisson algebra} is a triple $P = (P, \, m_P, \, [-, \,
-])$, where $(P, m_P)$ is an associative algebra, $(P, \, [-, \,
-])$ is a Lie algebra such that any Hamiltonian $[ -, \, r ] : P
\to P$ is a derivation of the associative algebra $(P, m_P)$, i.e.
the Leibniz law:
\begin{equation}\eqlabel{p1}
[pq, \, r ] = [p, \, r] \, q + p \, [q, \, r]
\end{equation}
holds for any $p$, $q$, $r\in P$. Usually, a Poisson algebra $P$
is by definition assumed to be commutative, as the main examples
are in classical differential geometry. However, throughout this
paper we do not impose this restriction. A morphism between two
Poisson algebras $P$ and $P'$ is a linear map $\varphi: P \to P'$
that is a morphism of associative algebras as well as of Lie
algebras.

\begin{remarks}\relabel{unitaPos}
$(1)$ Let $P$ be a unitary Poisson algebra. By applying
\equref{p1} for $p = q = 1_P$ we obtain:
\begin{equation}\eqlabel{p2}
[1_P, \, r ] = 0 = [ r, \, 1_P]
\end{equation}
for all $r \in P$. Any non unital Poisson algebra embeds into a
unital Poisson algebra by adding a unit $1_P$ to the associative
algebra $(P, m_P)$ and considering \equref{p2} as the defining
relation for the bracket evaluated in the unit $1_P$.

$(2)$ Let $P = (P, \, m_P)$ be an algebra and $u\in k$. Then $(P,
m_P, [-, \, -]_u)$ is a Poisson algebra, where $[a, \, b]_u := u
(ab - ba)$, for all $a$, $b\in P$ \cite{kobo}. In particular, any
associative algebra $P = (P, \, m_P)$ is a Poisson algebra with
the abelian Lie algebra structure, i.e. $[a, \, b] := 0$, for all
$a$, $b\in P$. Any Lie algebra $P = (P, \, [-, -])$ is a Poisson
algebra with the trivial algebra structure, i.e. $p q := 0$, for
all $p$, $q\in P$. Any vector space $V$ has a Poisson algebra
structure with the trivial algebra structure and the abelian Lie
algebra structure. This trivial Poisson algebra structure on $V$
will be called \emph{abelian} and will be denoted from now on by
$V_0$.

$(3)$ If $P$ is a Poisson algebra that is non-commutative and
\emph{prime} as an associative algebra, then \cite[Theorem
2.1]{Farkas} shows that the bracket on $P$ has the form $[p, q] =
\lambda (pq - qp)$, for some $\lambda \in {\mathcal Z}^{+} (P)$,
where ${\mathcal Z}^{+} (P)$ is the center of the ring of
quotients of $P$. In particular, any Poisson algebra structure on
the Weyl algebra $A_1$ has the Lie bracket of the form $[p, q] =
\lambda (pq - qp)$, for some $\lambda \in k$. For further results
on the structure of a bracket on a non-commutative Poisson algebra
we refer to \cite{kobo}.
\end{remarks}

\begin{examples} \exlabel{exemplemici}
$(1)$ Up to an isomorphism, there exist two $1$-dimensional
Poisson algebras: the first one is $k_0$, i.e. $k$ viewed with the
abelian Poisson algebra structure and the second one, denoted by
$k_1$, is the vector space $k$, with the unital algebra structure
($1 \cdot 1 := 1$) and the trivial bracket.

$(2)$ Any unital $2$-dimensional Poisson algebra has the trivial
bracket, thanks to the compatibility \equref{p2}. Therefore, their
classification follows from \cite[Corollary 4.5]{am-2013c} where
all $2$-dimensional associative, unital algebras over an arbitrary
field are classified. We point out that if $k$ has characteristic
$2$ then the number of types of $2$-dimensional Poisson algebras
can be infinite.
\end{examples}

Poisson bimodules will play a key role in constructing the
cohomolgical object ${\mathcal P} {\mathcal H}^{2} \, (P, \,
V_0)$. A \emph{Poisson bimodule} \cite{Flato, kobo2001} over a
Poisson algebra $P$ is a system $(V, \, \rightharpoonup, \,
\triangleleft, \, \triangleright)$ consisting of a vector space
$V$ and three bilinear maps such that $(V, \, \rightharpoonup, \,
\triangleleft) \in {}_P\Mm_P$ is a $P$-bimodule, $(V, \,
\triangleright) \in {}^P\Mm$ is a left Lie $P$-module satisfying
the following compatibility conditions for any $p$, $q\in P$ and
$x\in V$:
\begin{eqnarray}
(pq) \triangleright x &=& p \rightharpoonup (q \triangleright x) +
(p \triangleright x) \triangleleft q \eqlabel{bimod1}\\
\left[p, \, q \right] \rightharpoonup x &=& p \rightharpoonup (q
\triangleright x) - q \triangleright (p \rightharpoonup x) \eqlabel{bimod2}\\
x \triangleleft \left[p, \, q \right] &=& (q \triangleright x)
\triangleleft p - q \triangleright (x \triangleleft p).
\eqlabel{bimod3}
\end{eqnarray}
We denote by ${}_P^P\Mm_P$ the category of Poisson bimodules over
$P$ having as morphisms all linear maps which are compatible with
the actions. Any vector space $V$ has a trivial Poisson bimodule
structure over $P$, where $(\rightharpoonup, \, \triangleleft, \,
\triangleright)$ are all the trivial maps. Using the Leibniz law
\equref{p1}, we can easily see that $P := (P, \, \rightharpoonup
\, = \,  \triangleleft \, := \, m_P, \,\, \triangleright := \, [-,
-])$ is a Poisson bimodule over $P$.

\subsection*{The Hochschild product for associative algebras.}
We recall the construction of the Hochschild product for
associative algebras following the terminology used in
\cite{am-2013c}.

\begin{definition} \delabel{hocdat}
Let $P$ be an algebra and $V$ a vector space. A \emph{Hochschild
(or pre-crossed) datum} of $P$ by $V$ is a system $\Theta(P, V) =
(\rightharpoonup, \, \triangleleft, \, \vartheta, \, \cdot)$
consisting of four bilinear maps
$$
\rightharpoonup \, \, : P \times V \to V, \quad \triangleleft : V
\times P \to V, \quad \vartheta \,\, : P \times P \to V, \quad
\cdot = \cdot_{V} \, : V \times V \to V.
$$
Let $\Theta(P, V) = (\rightharpoonup, \, \triangleleft, \,
\vartheta, \, \cdot)$ be a Hochschild datum and we denote by $P
\star V = P \star_{(\triangleleft, \rightharpoonup, \vartheta,
\cdot)} V $ the vector space $P \times V$ with the multiplication
given by
\begin{equation} \eqlabel{hoproduct2}
(p, \, x) \star (q, \, y) := (pq, \, \vartheta (p, \, q) + p
\rightharpoonup y + x \triangleleft q + x \cdot y)
\end{equation}
for all $p$, $q\in P$, $x$, $y \in V$. $P \star V$ is called the
\emph{Hochschild (or crossed) product} associated to $\Theta(P,
V)$ if it is an associative algebra with the multiplication given
by \equref{hoproduct2}. In this case the Hochschild datum
$\Theta(P, V) = (\rightharpoonup, \, \triangleleft, \, \vartheta,
\, \cdot)$ is called a \emph{Hochschild (or crossed) system} of
$P$ by $V$ and we denote by ${\mathcal H} {\mathcal S} (P, \, V)$
the set of all Hochschild systems of $P$ by $V$.
\end{definition}

It was proved in \cite[Proposition 1.2]{am-2013c} that $\Theta(P,
V)$ is a Hochschild system if and only if the following
compatibility conditions hold for any $p$, $q$, $r\in P$ and $x$,
$y\in V$:
\begin{enumerate}
\item[(H0)] $(V, \cdot)$ is an associative algebra

\item[(H1)] $(x \cdot y) \lhd p = x \cdot (y \lhd p)$

\item[(H2)] $(x \lhd p) \cdot y = x \cdot (p \rightharpoonup y)$

\item[(H3)] $p \rightharpoonup (x \cdot y) = (p \rightharpoonup x)
\cdot y$

\item[(H4)] $(p \rightharpoonup x) \lhd q = p \rightharpoonup (x
\lhd q)$

\item[(H5)] $\vartheta(p, \, q) \lhd r = \vartheta(p, \, qr) -
\vartheta(pq, \, r) + p \rightharpoonup \vartheta(q, \, r)$

\item[(H6)] $(pq) \rightharpoonup x = p \rightharpoonup (q
\rightharpoonup x) - \vartheta(p, \, q) \cdot x$

\item[(H7)] $x \lhd (pq) = (x \lhd p) \lhd q - x \cdot
\vartheta(p, \, q)$.
\end{enumerate}
The Hochschild product was introduced in \cite[Theorem 6.2]{Hoch2}
in the special case when $\cdot$ is the trivial multiplication on
$V$ (i.e. $x\cdot y = 0$, for all $x$, $y\in V$). We point out
that in this case axioms (H0)- (H7) reduce to $(V,
\rightharpoonup, \triangleleft)$ being an $P$-bimodule and
$\vartheta : P \times P \to V$ being a $2$-cocycle.

If $P \star V$ is a Hochschild product, then the map $\pi_P : P
\star V \to P$, $\pi_P (p, x) := p$ is a surjective algebra
morphism and ${\rm Ker} (\pi) = 0 \times V \cong V$. Conversely,
\cite[Proposition 1.4]{am-2013c} proves the following: if $P$ is
an algebra, $E$ a vector space and $\pi : E \to P$ an epimorphism
of vector spaces with $V = {\rm Ker} (\pi)$, then any algebra
structure $\cdot_E$ which can be defined on the vector space $E$
such that $\pi : (E, \cdot_E) \to P$ becomes a morphism of
algebras is isomorphic to a Hochschild product $P \star V$ and
moreover, the isomorphism of algebras $ (E, \cdot_E) \cong P \star
V$ can be chosen such that it stabilizes $V$ and co-stabilizes
$P$.

\subsection*{The crossed product of Lie algebras.}
We shall recall the construction of the crossed product for Lie
algebras following the terminology of \cite[Section 3]{am-2013}.

\begin{definition} \delabel{crosslie}
Let $P = (P, \, [-,\, -])$ be a Lie algebra and $V$ a vector
space. A \emph{pre-crossed data} of $P$ by $V$ is a system $\Lambda
(P, V) = \bigl(\triangleright, \, f, [-, \, -]_V \bigl)$
consisting of three bilinear maps
$$
\triangleright \, : P \times V \to V, \quad f: P \times P \to V,
\quad [-, \, -]_V \, : V \times V \to V.
$$
Let $\Lambda (P, V) = \bigl(\triangleright, \, f, [-, \, -]_V \bigl)$
be a pre-crossed data of $P$ by $V$ and we denote by
$P \# \, V = P \#_{\triangleright}^f \, V$ the vector
space $P \times \, V $ with the bracket given for any $p$, $q \in
P$ and $x$, $y \in V$ by:
\begin{equation} \eqlabel{crossedliedef}
\{ (p, x), \, (q, y) \} := \bigl( [p, q], \, f(p, q) + p
\triangleright y - q \triangleright x + [x, \, y]_V \bigl).
\end{equation}
Then $P \# \, V$ is called the \emph{crossed
product} associated to $\Lambda (P, V) =
\bigl(\triangleright, \, f, [-, \, -]_V \bigl)$ if it is a Lie algebra with the
bracket \equref{crossedliedef}. In this case the pre-crossed data $\Lambda (P, V) =
\bigl(\triangleright, \, f, [-, \, -]_V \bigl)$ is called a
\emph{crossed system} of $P$ by $V$ and we denote by
${\mathcal L} {\mathcal S} (P, \, V)$ the set of all crossed
systems of the Lie algebra $P$ by $V$.
\end{definition}

As a special case of \cite[Theorem 2.2]{am-2013} or by a
straightforward computation it is easy to see that $\Lambda (P, V)
= \bigl(\triangleright, \, f, [-, \, -]_V \bigl)$ is a crossed
system of $P$ by $V$ if and only if the following compatibilities
hold for any $p$, $q$, $r \in P$ and $x$, $y\in V$:
\begin{enumerate}
\item[(L0)] $(V, [-, \, -]_V)$ is a Lie algebra

\item[(L1)] $ f(p, p) = 0 $

\item[(L2)] $ p \triangleright [x, \, y]_V = [p \triangleright x ,
\, y]_V + [x, \, p \triangleright y]_V$

\item[(L3)] $[p, \, q] \triangleright x = p \triangleright (q
\triangleright x) - q\triangleright (p \triangleright x) + [x, \,
f (p, q)]_V$

\item[(L4)] $f(p, [q, \, r ] ) + f(q, [r, \, p ]) + f(r, [p, \, q]
) + p\triangleright f(q, r) + q\triangleright f(r, p) + r
\triangleright f(p, q) = 0$.
\end{enumerate}

\begin{example}\exlabel{licoflag}
Let $P$ be a Lie algebra having $\{e_i \, | \, i\in I \}$ as a
basis. Then there exists a bijection between the set ${\mathcal L}
{\mathcal S} (P, \, k)$ of all crossed systems of $P$ by $k$ and
the set of pairs $(\lambda, f)$ consisting of a linear map
$\lambda : P \to k$ and a bilinear map $f : P \times P \to k$
satisfying the following compatibility conditions for all $p$,
$q$, $r\in P$:
\begin{eqnarray*}
&& f (p, p) = \lambda ([p, \, q]) = 0 \\
&& f(p, [q, \, r ] ) + f(q, [r, \, p ]) + f(r, [p, \, q] ) +
\lambda (p) f(q, r) + \lambda (q) f(r, p) + \lambda (r) f(p, q) =
0.
\end{eqnarray*}
Under this bijection the crossed product $P \#_{\lambda}^f \, k$
corresponding to $(\lambda, f)$ is the Lie algebra having $\{x, \,
e_i \, | \, i\in I \}$ as a basis and the bracket $\left[-, \, -
\right]_{\lambda}^f$ defined for any $i$, $j\in I$ by:
\begin{eqnarray*}
\left[e_i, \, e_j \right]_{\lambda}^f :=  \left[e_i, \, e_j
\right] + f(e_i, \, e_j) \, x, \quad \left[e_i, \, x
\right]_{\lambda}^f := \lambda (e_i) \, x.
\end{eqnarray*}

Indeed, since $V := k$ the Lie bracket on $k$ is the trivial map
$[-, -] = 0$. Moreover, any bilinear map $\triangleright : P
\times k \to k$ is implemented by a unique linear map $\lambda : P
\to k$ such that $p \triangleright 1_k = \lambda (p)$. The rest of
the proof is straightforward.
\end{example}

\section{Crossed products and the global extension problem}\selabel{unifiedprod}
In this section we shall give the theoretical answer to the
GE-problem. First we introduce the following:

\begin{definition}\delabel{exdatum}
Let $P$ be a Poisson algebra and $V$ a vector space. A
\textit{pre-crossed datum of $P$ by $V$} is a system $\Omega(P, V)
= \bigl(\rightharpoonup, \, \triangleleft, \, \vartheta, \,
\cdot_V, \, \triangleright, \, f, [-, \, -]_V \bigl)$ consisting
of seven bilinear maps
$$
\rightharpoonup \, \, : P \times V \to V, \quad \triangleleft : V
\times P \to V, \quad \vartheta \,\, : P \times P \to V, \quad
\cdot_V \, : V \times V \to V
$$
$$
\triangleright \, : P \times V \to V, \quad f: P \times P \to V,
\quad [-, \, -]_V \, : V \times V \to V.
$$
Let $\Omega(P, V) = \bigl(\rightharpoonup, \, \triangleleft, \,
\vartheta, \, \cdot_V, \, \triangleright, \, f, [-, \, -]_V
\bigl)$ be a pre-crossed datum of $P$ by $V$. We denote by $ P \,
\sharp_{\Omega(P, V)} \, V = P \, \sharp \, V$ the vector space $P
\, \times V$ together with the multiplication $\star$ given by
\equref{hoproduct2} and the bracket $\{-, \, -\}$ given by
\equref{crossedliedef}, i.e.
\begin{eqnarray}
(p, x) \star (q, y) &:=& (p q, \,\, \vartheta (p, q) + p
\rightharpoonup y + x \triangleleft q + x \cdot_V y)
\eqlabel{crpos1} \\
\{ (p, x), \, (q, y) \} &:=& \bigl( [p, q], \,\, f(p, q) + p
\triangleright y - q \triangleright x + [x, \, y]_V \bigl).
\eqlabel{crpos2}
\end{eqnarray}
for all $p$, $q\in P$, $x$, $y \in V$. The object $P \, \sharp \,
V = (P \times V, \, \star, \{-, \, - \})$ is called the
\textit{crossed product} of $P$ and $\Omega(P, V)$ if it is a
Poisson algebra. In this case the pre-crossed datum $\Omega(P, V)
$ is called a \textit{crossed system} of the Poisson algebra $P$
by $V$. The maps $\leftharpoonup$, $\triangleleft$,
$\triangleright$ are called the \textit{actions} of $\Omega(P, V)$
while $\vartheta$ and $f$ are called the \textit{cocycles} of
$\Omega(P, V)$.
\end{definition}

The next theorem provides the necessary and sufficient conditions
that need to be fulfilled by a pre-crossed datum $\Omega(P, V)$
such that $P \sharp V$ is a Poisson algebra.

\begin{theorem}\thlabel{1}
Let $\Omega(P, V) = \bigl(\rightharpoonup, \, \triangleleft, \,
\vartheta, \, \cdot_V, \, \triangleright, \, f, [-, \, -]_V
\bigl)$ be a pre-crossed datum of a Poisson algebra $P$ by a
vector space $V$. Then $P \sharp \, V$ is a Poisson algebra if and
only if the following compatibilities hold:
\begin{enumerate}
\item[(P0)] $\bigl(\rightharpoonup, \, \triangleleft, \,
\vartheta, \, \cdot_V \bigl)$ is a Hochschild system of the
algebra $P$ by $V$, $\bigl(\triangleright, \, f, [-, \, -]_V
\bigl)$ is a crossed system of the Lie algebra $P$ by $V$ and $(V,
\, \cdot_V, \, [-, \, -]_V)$ is a Poisson algebra

\item[(P1)] $f (pq, \, r) - f(p,\, r) \triangleleft q - \, p
\rightharpoonup f(q, \, r) = r \triangleright \vartheta(p, \, q) +
\vartheta ([p, \, r], \, q) + \vartheta (p, \, [q, \, r])$

\item[(P2)] $ (pq) \triangleright x = p \rightharpoonup (q
\triangleright x) + (p \triangleright x) \triangleleft q -
[\vartheta(p, \, q), \, x]_V$

\item[(P3)] $[p, \, q] \rightharpoonup x = p \rightharpoonup (q
\triangleright x) - q \triangleright (p \rightharpoonup x) - f(p,
\, q) \cdot_V  x $

\item[(P4)] $p \rightharpoonup [x, \, y]_V = [p \rightharpoonup x,
\, y]_V - (p \triangleright y)\cdot_V x $

\item[(P5)] $x \triangleleft [p, \, q] = (q \triangleright x)
\triangleleft p - q \triangleright (x \triangleleft p) - x \cdot_V
f(p, \, q)$

\item[(P6)] $ [x, \, y]_{V} \triangleleft p = [x \triangleleft p,
\, y]_V - x \cdot_V (p \triangleright y)$

\item[(P7)] $ p \triangleright (x \cdot_V y) = (p \triangleright x
) \cdot_V y + x \cdot_V (p \triangleright y)$
\end{enumerate}
for all $p$, $q$, $r \in P$ and $x$, $y \in V$.
\end{theorem}

\begin{proof}
We have already noticed in Preliminaries that $(P \sharp \, V,
\star)$ is an associative algebra if and only if
$\bigl(\rightharpoonup, \, \triangleleft, \, \vartheta, \, \cdot_V
\bigl)$ is a Hochschild system of the associative algebra $P$ by
$V$ and $(P \sharp \, V, \{-, \, -\})$ is a Lie algebra if and
only $\bigl(\triangleright, \, f, [-, \, -]_V \bigl)$ is a crossed
system of the Lie algebra $P$ by $V$. These are the first two
assumptions from (P0) which from now on we assume to be fulfilled.
Then, $(P \sharp \, V, \star, \{-, \, -\})$ is a Poisson algebra
if and only if the Leibniz identity holds, i.e.
\begin{equation}\eqlabel{005}
\{ (p, x) \star (q, y), \, (r, z) \} = \{(p, x), \, (r, z)\} \star
(q, y) + (p, x) \star \{(q, y), \, (r, z) \}
\end{equation}
for all $p$, $q$, $r\in P$ and $x$, $y$, $z\in V$. The rest of the
proof relies on a detailed analysis of the Leibniz identity
\equref{005}: since in $P \times V$ we have $(p, x) = (p, 0) + (0,
x)$ it follows that \equref{005} holds if and only if it holds for
all generators of $P \times V$, i.e. for the set $\{(p, \, 0) ~|~
p \in P\} \cup \{(0, \, x) ~|~ x \in V\}$. However, since the
computations are rather long but straightforward we will only
indicate the main steps of the proof, the details being left to
the reader.

We will start by proving that \equref{005} holds for the triple
$(p, 0)$, $(q, 0)$, $(r, 0)$ if and only if (P1) holds. Indeed, we
can easily see that the left hand side of \equref{005}, evaluated
at $(p, 0)$, $(q, 0)$, $(r, 0)$ is equal to $\bigl( [pq, \, r], \,
f(pq, \, r) - r \triangleright \vartheta (p, q) \bigl)$ while the
right hand side of \equref{005} is
$$
\bigl( [p, r] q + p [q, r], \, \vartheta ([p, \, r], \, q) + f (p,
r) \triangleleft q +  \vartheta (p, \, [q, \, r]) + p
\rightharpoonup f(q, r) \bigl).
$$
Since $P$ is a Poisson algebra we obtain that \equref{005} holds
for the triple $(p, 0)$, $(q, 0)$, $(r, 0)$ if and only if (P1)
hold.

In the same manner we can prove the following: \equref{005} holds
for the triple $(p, 0)$, $(q, 0)$, $(0, x)$ if and only if (P2)
holds; \equref{005} holds for the triple $(p, 0)$, $(0, x)$, $(q,
0)$ if and only if (P3) holds; \equref{005} holds for the triple
$(p, 0)$, $(0, x)$, $(0, y)$ if and only if (P4) holds;
\equref{005} holds for the triple $(0, x)$, $(p, 0)$, $(q, 0)$ if
and only if (P5) holds; \equref{005} holds for the triple $(0,
x)$, $(p, 0)$, $(0, y)$ if and only if (P6) holds and \equref{005}
holds for the triple $(0, x)$, $(0, y)$, $(p, 0)$ if and only if
(P7) holds. Finally, one can see that \equref{005} holds for the
triple $(0, x)$, $(0, y)$, $(0, z)$ if and only if
$$
[x \cdot_V y, \, z]_V = [x, \, z]_V \cdot_V y + x \cdot_V [y, \,
z]_V
$$
i.e. the Leibniz low holds for $(V, \cdot_V, [-, -]_V)$, that is
$(V, \, \cdot_V, \, [-, \, -]_V)$ is a Poisson algebra and this
finishes the proof.
\end{proof}

From now on, a crossed system of a Poisson algebra $P$ by a vector
space $V$ will be viewed as a pre-crossed datum $\Omega(P, V) =
\bigl(\rightharpoonup, \, \triangleleft, \, \vartheta, \, \cdot_V,
\, \triangleright, \, f, [-, \, -]_V \bigl)$ satisfying the
compatibility conditions (P0)-(P7) of \thref{1}. We denote by
${\mathcal P} {\mathcal S} (P, V)$ the set of all crossed systems
of $P$ by $V$. We also use the following convention: if one of the
maps of a crossed system $\Omega(P, V) = \bigl(\rightharpoonup, \,
\triangleleft, \, \vartheta, \, \cdot_V, \, \triangleright, \, f,
[-, \, -]_V \bigl)$ is the trivial one, then we will omit it from
the system $\Omega(P, V)$.

\begin{example}\exlabel{twistedproduct}
Let $\Omega(P, V) = \bigl(\rightharpoonup, \, \triangleleft, \,
\vartheta, \, \cdot_V, \, \triangleright, \, f, [-, \, -]_V
\bigl)$ be a pre-crossed datum of a Poisson algebra $P$ by $V$
such that $\rightharpoonup$, $\triangleleft$, $\vartheta$,
$\triangleright$, $f$ are all the trivial maps. Then $\Omega(P,
V)$ is a crossed system of $P$ by $V$ if and only if $(V, \,
\cdot_V, \, [-, \, -]_V \bigl)$ is a Poisson algebra. The
associated crossed product $P \, \sharp \, V$ is just the direct
product $P \times V$ of Poisson algebras.
\end{example}

Any crossed product $P \, \sharp \, V$ is an extension of the Poisson algebra
$P$ by $V$ via the canonical maps
\begin{eqnarray} \eqlabel{extenho1}
\xymatrix{ 0 \ar[r] & V \ar[r]^{i_{V}} & P \, \sharp \,
V\ar[r]^{\pi_{P}} & P \ar[r] & 0 }
\end{eqnarray}
where $i_V (v) = (0, v)$ and $\pi_P (p, x) := p$. Conversely, the
crossed product is the tool to answer the global extension problem
for Poisson algebras. More precisely, the next result provides the
answer to the description part of the GE-problem and can be seen
as a generalization at the level of Poisson algebras of
\cite[Theorem 5.2]{Hoch2}:

\begin{proposition}\prlabel{hocechiv}
Let $P$ be a Poisson algebra, $E$ a vector space and $\pi : E \to
P$ an epimorphism of vector spaces with $V = {\rm Ker} (\pi)$.
Then any Poisson algebra structure $(\cdot_E, \, [-, \, -]_E)$
which can be defined on $E$ such that $\pi : (E, \, \cdot_E, \,
[-, \, -]_E) \to P$ is a morphism of Poisson algebras is
isomorphic to a crossed product $P \, \sharp \, V$ and moreover,
the isomorphism of Poisson algebras $ (E, \, \cdot_E, \, [-, \,
-]_E) \cong  P \, \sharp \, V$ can be chosen such that it
stabilizes $V$ and co-stabilizes $P$.

In particular, any Poisson algebra extension of $P$ by $V$ is
cohomologous to a crossed product extension \equref{extenho1}.
\end{proposition}

\begin{proof}
Indeed, let $(\cdot_E, \, [-, \, -]_E)$ be a Poisson algebra
structure of $E$ such that $\pi: (E, \, \cdot_E, \, [-, \, -]_E)
\to P$ is a morphism of Poisson algebras. Let $s : P \to E$ be a
$k$-linear section of $\pi$, i.e. $\pi
\circ s = {\rm Id}_{P}$. Using this section $s$ we define a
pre-crossed datum $\Omega(P, V) = \bigl(\rightharpoonup =
\rightharpoonup_s, \, \triangleleft = \triangleleft_s, \,
\vartheta = \vartheta_s, \, \cdot_V = \cdot_{V, s}, \,
\triangleright = \triangleright_s, \, f=f_s, [-, \, -]_V = [-, \,
-]_{V, s} \bigl)$ of $P$ by $V$ by the following formulas:
\begin{eqnarray*}
&& \rightharpoonup : P \times V \to V, \quad \,\,\,\,\,\,\, p
\rightharpoonup x := s(p) \cdot_E \, x, \quad \,\,\, \triangleleft
: V \times P \to V, \quad x \triangleleft p := x \cdot_E \, s(p)\\
&& \vartheta  \, : P \times P \to V, \,\, \vartheta (p, \, q) :=
s(p) \cdot_E \, s(q) - s(pq), \quad \cdot_V : V \times V
\to V, \,\,  x \cdot_V \, y := x \cdot_E \, y \\
&& \triangleright : P \times V \to V, \,\, p \triangleright x :=
[s(p), \, x]_E, \quad [-, \, -]_V : V \times V \to V, \,\, [x, \, y]_V := [x, \, y]_E \\
&& f : P \times P \to V, \quad  f(p, q) := [s(p), \, s(q)]_E - s
\bigl( [p, \, q] \bigl)
\end{eqnarray*}
for all $p$, $q\in P$ and $x$, $y\in V$. Then
$$
\varphi : P \times V \to E, \qquad \varphi (p, x) := s(p) + x
$$
is an isomorphism of vector spaces with the inverse $\varphi^{-1}
(y) = (\pi(y), \, y - s (\pi(y)) )$, for all $y\in E$. The key
step is the following: the unique Poisson algebra structure
$(\star, \, \{-, \, -\}) $ that can be defined on the direct
product of vector spaces $P \times V$ such that $\varphi : P
\times V \to (E, \cdot_E, [-, \, -]_E)$ is an isomorphism of
Poisson algebras has the multiplication $\star$ given by
\equref{crpos1} and the bracket $\{-, \, -\}$ given by
\equref{crpos2} associated to the system $\bigl(\rightharpoonup,
\, \triangleleft, \, \vartheta, \, \cdot_V, \, \triangleright, \,
f, \, [-, \, -]_V \bigl)$ as defined above arising from $s$.
Indeed, \cite[Proposition 1.4]{am-2013c} proves that the unique
multiplication $\star$ that can be defined on $P \times V$ such
that $\varphi: P \times V \to (E, \cdot_E)$ is an isomorphism of
algebras is the one given by \equref{crpos1} associated to the
Hochschild system $(\rightharpoonup_s, \, \triangleleft_s, \,
\vartheta_s, \, \cdot_{V, s})$. In a similar manner we can show
that the unique bracket $\{-, \, - \}$ that can be defined on $P
\times V$ such that $\varphi: P \times V \to (E, [-, \, -]_E)$ is
an isomorphism of Lie algebras is the one given by \equref{crpos2}
associated to the crossed system $(\triangleright_s, \, f_s, [-,
\, -]_{V, s})$ of the Lie algebra $P$ by $V$. Thus, $\varphi : P
\, \sharp \, V \to (E, \cdot_E, [-, \, -]_E)$ is an isomorphism of
Poisson algebras that stabilizes $V$ and co-stabilizes $P$.
\end{proof}

\begin{example}
Among all special cases of crossed products the most important is
the semidirect product since, exactly as in the case of groups,
Lie algebras or associative algebras, the semidirect products will
describe the split epimorphisms in the category of Poisson
algebras. Let $P$ and $V = (V, \, \cdot_V, \, [-, \, -]_V)$ be two
given Poisson algebras. Then a \emph{semi-direct system} of $P$ by
$V$ is a system consisting of three bilinear maps
$\bigl(\rightharpoonup, \, \triangleleft, \, \triangleright
\bigl)$ such that $\bigl(\rightharpoonup, \, \triangleleft, \,
\vartheta := 0, \, \cdot_V, \, \triangleright, \, f := 0, \, [-,
\, -]_V \bigl) = \bigl(\rightharpoonup, \, \triangleleft, \,
\cdot_V, \, \triangleright, \, [-, \, -]_V \bigl)$ is a crossed
system of $P$ by $V$. If $\bigl(\rightharpoonup, \, \triangleleft,
\, \triangleright \bigl)$ is a semi-direct system between $P$ and
$V$, then the associated crossed product $P \, \sharp \, V$ is
denoted by $P \rtimes V$ and is called the \emph{semi-direct
product} of $P$ and $V$. Thus, $P \rtimes V$ is the vector space
$P \times V$ with the Poisson algebra structure defined by:
\begin{eqnarray}
(p, x) \star (q, y) &:=& (p q, \, p \rightharpoonup y + x
\triangleleft q + x \cdot_V y)
\eqlabel{semi1} \\
\{ (p, x), \, (q, y) \} &:=& \bigl( [p, q], \, p \triangleright y
- q \triangleright x + [x, \, y]_V \bigl) \eqlabel{semi2}
\end{eqnarray}
for all $p$, $q\in P$, $x$, $y \in V$. If $P \rtimes V$ is a
semidirect product of Poisson algebras then the canonical
projection $\pi_P : P \rtimes V \to P$, $\pi_P (p, x) := p$ is a
morphism of Poisson algebras that has a section $s_P : P \to P
\rtimes V$, $s_P (p) := (p, 0)$ that is also a morphism of Poisson
algebras. Conversely, the semidirect product of Poisson algebras
describes the split epimorphisms in the category of Poisson
algebras. Indeed, let $\pi: E \to P$ be a morphism of Poisson
algebras which has a section that is a Poisson algebra map. Then
there exists an isomorphism of Poisson algebras $E \cong P \rtimes
V$, where $P \rtimes V$ is the semidirect product between $P$ and
$V = {\rm Ker} (\pi)$. The result follows from the proof of
\prref{hocechiv}: if $s: P \to E$ is a morphism of Poisson
algebras, then the cocycles $f = f_s$ and $\vartheta =
\vartheta_s$ constructed there are both the trivial maps, i.e. the
corresponding crossed product $P \, \sharp \, V$ is a semidirect
product $P \rtimes V$.
\end{example}

Based on \prref{hocechiv} the answer to the classification part of
the GE-problem reduces to the classification of all crossed
products associated to all crossed systems between $P$ and $V$.
This is what we do next by explicitly constructing a
classification object, denoted by ${\mathcal G} {\mathcal P}
{\mathcal H}^{2} \, (P, \, V)$. First we need the following
technical result:

\begin{lemma}\lelabel{HHH}
Let $P$ be a Poisson algebra, $\Omega(P, V) =
\bigl(\rightharpoonup, \, \triangleleft, \, \vartheta, \, \cdot_V,
\, \triangleright, \, f, [-, \, -]_V \bigl)$ and $\Omega '(P, V) =
\bigl(\rightharpoonup', \, \triangleleft', \, \vartheta', \,
\cdot_V', \, \triangleright', \, f', [-, \, -]_V' \bigl)$ two
crossed systems of $P$ by a vector space $V$ and $P \, \sharp \,
V$, respectively $P \, \sharp ' \, V$, the corresponding crossed
products. Then there exists a bijection between the set of all
morphisms of Poisson algebras $\psi: P \, \sharp \, V \to P \,
\sharp ' \, V$ which stabilize $V$ and co-stabilize $P$ and the
set of all linear maps $r: P \to V$ satisfying the following
compatibilities for all $p$, $q \in P$, $x$, $y \in V$:
\begin{enumerate}
\item[(M1)] $x \cdot_V y = x \cdot_V ' y$

\item[(M2)] $x \triangleleft p = x \triangleleft ' p + x \cdot_V '
r(p)$

\item[(M3)] $p \rightharpoonup x = p \rightharpoonup ' x + r(p)
\cdot_V ' \, x$

\item[(M4)] $\vartheta(p, \, q) = \vartheta '(p, \, q) + p
\rightharpoonup' r(q) + r(p) \lhd ' q - r(p q)  + r(p) \cdot_V '
r(q)$

\item[(M5)] $[x, \, y]_V = [x, \, y]_V'$

\item[(M6)] $p \triangleright x = p \triangleright' x + [r(p), \,
x]_V'$

\item[(M7)] $f (p, \, q) = f' (p, \, q) + p \triangleright' r(q) -
q \triangleright' r(p) + [r(p), \, r(q)]_V' - r \bigl([p, \,
q]\bigl)$.
\end{enumerate}
Under the above bijection the morphism of Poisson algebras $\psi =
\psi_{r}: P \, \sharp \, V \to P \, \sharp ' \, V$ corresponding
to $r: P \to V$ is given by $ \psi(p, x) = (p, \, r(p) + x)$, for
all $p \in P$ and $x \in V$. Moreover, $\psi = \psi_{r}$ is an
isomorphism with the inverse $\psi^{-1}_{r} = \psi_{-r}$.
\end{lemma}

\begin{proof} A linear map $\psi: P \, \sharp \, V \to P \, \sharp ' \, V$
stabilizes $V$ and co-stabilizes $P$ if and only if there exists a
uniquely determined linear map $r: P \to V$ such that $\psi(p, \,
x) = (p, r(p) + x)$, for all $p \in P$, $x \in V$. Let $\psi =
\psi_{r}$ be such a linear map. Then $\psi : P \, \sharp \, V \to
P \, \sharp ' \, V$ is a morphism of associative algebras if and
only if the compatibility conditions (M1)- (M4) hold (\cite[Lemma
1.5]{am-2013c}). By a straightforward computation we can also
prove that $\psi = \psi_{r}: P \, \sharp \, V \to P \, \sharp ' \,
V$ is a morphism of Lie algebras if and only if (M5)-(M7) hold.
For this it is enough to check the compatibility condition $\psi
\Bigl( \{ (p, x), \, (q, y) \}  \Bigl) = \{ \psi(p, \, x), \,
\psi(q, \, y) \}'$ on the set of generators, i.e. for the set
$\{(p, \, 0) ~|~ p \in P\} \cup \{(0, \, x) ~|~ x \in V\}$.
\end{proof}

\leref{HHH} leads to the following definition:

\begin{definition}\delabel{echiaa}
Let $P$ be a Poisson algebra and $V$ a vector space. Two crossed
systems $\Omega(P, V) = \bigl(\rightharpoonup, \, \triangleleft,
\, \vartheta, \, \cdot_V, \, \triangleright, \, f, [-, \, -]_V
\bigl)$ and $\Omega '(P, V) = \bigl(\rightharpoonup', \,
\triangleleft', \, \vartheta', \, \cdot_V', \, \triangleright', \,
f', [-, \, -]_V' \bigl)$ are called \emph{cohomologous}, and we
denote this by $\Omega (P, V) \approx \Omega '(P, V)$, if and only
if $\cdot_V = \cdot_V '$, $[-, \, -] = [-, \, -]_V'$ and there
exists a linear map $r: P \to V$ such that for any $p$, $q \in P$,
$x$, $y \in V$ we have:
\begin{eqnarray}
p \rightharpoonup x &=& p \rightharpoonup ' x + r(p) \cdot_V  \,
x \eqlabel{coc1}\\
x \triangleleft p &=& x \triangleleft ' p + x \cdot_V  r(p) \eqlabel{coc2} \\
\vartheta(p, \, q) &=& \vartheta '(p, \, q) + p \rightharpoonup'
r(q) + r(p) \lhd ' q - r(p q)  + r(p) \cdot_V  r(q) \eqlabel{coc3}\\
p \triangleright x &=& p \triangleright' x + [r(p), \, x]_V \eqlabel{coc4}\\
f (p, \, q) &=& f' (p, \, q) + p \triangleright' r(q) - q
\triangleright' r(p) + [r(p), \, r(q)]_V - r \bigl([p, \,
q]\bigl). \eqlabel{coc5}
\end{eqnarray}
\end{definition}

\begin{example}\exlabel{coboundary}
Let $P$ and $V$ be two Poisson algebras and $\Omega(P, V) =
\bigl(\cdot_V, \, [-, \, -]_V \bigl)$ be the trivial crossed
system of the Poisson algebra $P$ by $V$ from
\exref{twistedproduct}. A crossed system $\Omega(P, V) =
\bigl(\rightharpoonup, \, \triangleleft, \, \vartheta, \, \cdot_V,
\, \triangleright, \, f, [-, \, -]_V \bigl)$ cohomologous with the
trivial crossed system is called a \emph{coboundary}. Hence,
$\Omega(P, V) = \bigl(\rightharpoonup, \, \triangleleft, \,
\vartheta, \, \cdot_V, \, \triangleright, \, f, [-, \, -]_V
\bigl)$ is a coboundary if and only if there exists a linear map
$r: P \to V$ such that $\rightharpoonup$, $\triangleleft$,
$\vartheta$, $\triangleright$ and $f$ are implemented by $r$ via
the formulas:
\begin{eqnarray*}
p \rightharpoonup x &=& r(p) \cdot_V  \, x, \quad x \triangleleft p = x \cdot_V  r(p),
\quad p \triangleright x = [r(p), \, x]_V  \\
\vartheta(p, \, q) &=& r(p) \cdot_V  r(q) - r(p q), \quad f (p, \, q) =  [r(p), \, r(q)]_V - r \bigl([p, \, q]\bigl)
\end{eqnarray*}
for all $p$, $q \in P$, $x$, $y \in V$. \leref{HHH} shows that any
crossed product $P \, \sharp \, V$ associated to a coboundary is
isomorphic to the usual direct product $P \times V$ of Poisson
algebras.
\end{example}

As a conclusion, we obtain the theoretical answer to the GE-problem
for Poisson algebras:

\begin{theorem}\thlabel{main1222}
Let $P$ be a Poisson algebra, $E$ a vector space and $\pi : E \to
P$ an epimorphism of vector spaces with $V = {\rm Ker} (\pi)$.
Then $\approx$ is an equivalence relation on the set ${\mathcal P}
{\mathcal S} (P, V)$ of all crossed systems of $P$ by $V$. If we
denote by ${\mathcal G} {\mathcal P} {\mathcal H}^{2} \, (P, \, V)
:= {\mathcal P} {\mathcal S} (P, V)/ \approx $, then the map
$$
{\mathcal G} {\mathcal P} {\mathcal H}^{2} \, (P, \, V) \to {\rm
Gext} \, (E, P), \,\,\, \overline{ (\rightharpoonup,
\triangleleft, \vartheta, \cdot_V, \, \triangleright,  f, [-, \,
-]_V \bigl) } \, \longmapsto \, P \, \sharp_{(\rightharpoonup, \,
\triangleleft, \, \vartheta, \, \cdot_V, \, \triangleright, \, f,
[-, \, -]_V \bigl)} \, V
$$
is a bijection between ${\mathcal G} {\mathcal P}{\mathcal H}^{2}
\, (P, \, V)$ and ${\rm Gext} (E, P)$.
\end{theorem}

\begin{proof} Follows from \thref{1}, \prref{hocechiv} and \leref{HHH}.
\end{proof}

\subsection*{The global vs. the local extension problem.}
Computing the classifying object ${\mathcal G} {\mathcal P}
{\mathcal H}^{2} \, (P, \, V)$ is a highly nontrivial problem. We
explain below the details as well as the connection with the
classical (i.e. local) extension problem. The extension problem
has as input data two given Poisson algebras $P$ and $V = (V,
\cdot_V, [-, \, -]_V)$ and it asks for the classification of all
extensions of $P$ by the fixed Poisson algebra $V$. We denote by
${\rm Ext} \, (P, (V, \cdot_V, [-, \, -]_V))$ the isomorphism
classes of all extensions of $P$ by $V$, i.e. up to an isomorphism
of Poisson algebras that stabilizes $V$ and co-stabilizes $P$. We
should point out that we work in the general case, with no further
assumptions on the Poisson algebras (e.g. commutative, abelian
etc.). The answer to the extension problem follows as a special
case of \thref{main1222}:

Let ${\mathcal L} {\mathcal P} {\mathcal S} (P, (V, \cdot_V, [-,
\, -]_V))$ be the set of all \emph{local crossed systems} of $P$
by $(V, \cdot_V, [-, \, -]_V)$, i.e. ${\mathcal L} {\mathcal P}
{\mathcal S} (P, (V, \cdot_V, [-, \, -]_V))$ contains the set of
all $5$-uples $\bigl(\rightharpoonup, \, \triangleleft, \,
\vartheta, \, \triangleright, \, f \bigl)$ of bilinear maps
$$
\rightharpoonup \, \, : P \times V \to V, \quad \triangleleft : V
\times P \to V, \quad \vartheta \,\, : P \times P \to V, \quad
\triangleright \, : P \times V \to V, \quad f: P \times P \to V,
$$
satisfying axioms (H1)-(H7), (L1)-(L4) and (P1)-(P7). Two local
crossed systems $\bigl(\rightharpoonup, \, \triangleleft, \,
\vartheta, \, \triangleright, \, f \bigl)$ and
$\bigl(\rightharpoonup', \, \triangleleft', \, \vartheta', \,
\triangleright', \, f' \bigl)$ are \emph{local cohomologous} and
we denote this by $\bigl(\rightharpoonup, \, \triangleleft, \,
\vartheta, \, \triangleright, \, f \bigl) \, \approx_l \,
\bigl(\rightharpoonup', \, \triangleleft', \, \vartheta', \,
\triangleright', \, f' \bigl)$ if there exists a linear map $r: P
\to V$ satisfying the compatibility conditions
\equref{coc1}-\equref{coc5}. The answer in the general case to the
extension problem for Poisson algebras and the connection
\emph{local vs. global} is given below. In particular, we obtain
the formula for computing ${\mathcal G} {\mathcal P} {\mathcal
H}^{2} \, (P, \, V)$:

\begin{corollary}\colabel{formulamare}
Let $P$ be a Poisson algebra. Then:

$(1)$ If $V = (V, \cdot_V, [-, \, -]_V)$ is a fixed Poisson
algebra, then $\approx_l$ is an equivalence relation on the set
${\mathcal L} {\mathcal P} {\mathcal S} (P, (V, \cdot_V, [-, \,
-]_V))$ of all local crossed systems of $P$ by $(V, \cdot_V, [-,
\, -]_V)$. If we denote by ${\mathcal P} {\mathcal H}^{2} \, (P,
\, (V, \cdot_V, [-, \, -]_V))$ the quotient set ${\mathcal L}
{\mathcal P} {\mathcal S} (P, (V, \cdot_V, [-, \, -]_V))/
\approx_l$, then the map
$$
{\mathcal P} {\mathcal H}^{2} \, (P, \, (V, \cdot_V, [-, \, -]_V))
\to {\rm Ext} \, \bigl(P, (V, \cdot_V, [-, \, -]_V)\bigl), \,\,\, \overline{
(\rightharpoonup, \triangleleft, \vartheta, \triangleright,
f\bigl) }  \longmapsto \, P \, \sharp_{(\rightharpoonup, \,
\triangleleft, \vartheta, \triangleright, f\bigl)} \, V
$$
is a bijection.

$(2)$ Let $E$ be a vector space and $\pi : E \to P$ an
epimorphism of vector spaces with $V = {\rm Ker} (\pi)$. Then:
\begin{equation}\eqlabel{balsoi}
{\mathcal G} {\mathcal P} {\mathcal H}^{2} \, (P, \, V) \, = \,
\sqcup_{( \cdot_V, \{-, -\}_V )} \, {\mathcal P} {\mathcal H}^{2}
\, (P, \, (V, \cdot_V, \{-, -\}_V))
\end{equation}
where the coproduct in the right hand side is in the category of
sets over all possible Poisson algebra structures $(\cdot_V, \{-,
-\}_V )$ on the vector space $V$.
\end{corollary}

\begin{proof}
The first part follows from the above considerations and
\thref{main1222} while the second part follows from relations (M1)
and (M5) of \leref{HHH}.
\end{proof}

\subsection*{The abelian case. Metabelian Poisson algebras.}
The formula \equref{balsoi} for computing ${\mathcal G} {\mathcal
P} {\mathcal H}^{2} \, (P, \, V)$ highlights the difficulty of the
GE-problem. In fact, even computing every cohomology object
${\mathcal P} {\mathcal H}^{2} \, (P, \, (V, \cdot_V, \{-,
-\}_V))$ from the right hand side, for a given Poisson algebra
$(V, \cdot_V, \{-, -\}_V)$ is a problem far from being trivial.
Traditionally, regardless if we consider the case of groups
\cite{R}, associative algebras \cite{Hoch2} or Lie algebras
\cite{CE}, the extension problem and therefore the cohomology
groups are considered only in the abelian case. In other words,
only one of the elements in the coproduct from the right hand side
of \equref{balsoi} is well understood, namely the one
corresponding to the abelian case -- which for Poisson algebras
comes down to $V$ being an abelian Poisson algebra, i.e. $x
\cdot_V y = [x, \, y]_V = 0$, for all $x$, $y \in V$. We recall
that we have denoted by $V_0$ the vector space $V$ with the
abelian Poisson algebra structure.

In what follows we will describe the object ${\mathcal P}
{\mathcal H}^{2} \, (P, \, V_0)$, for the abelian Poisson algebra
$V_0$. In this case the set ${\mathcal L} {\mathcal P} {\mathcal
S} (P, \, V_0)$ of local crossed systems of $P$ by $V_0$ consists
of the set of all $5$-uples $\bigl(\rightharpoonup, \,
\triangleleft, \, \vartheta, \, \triangleright, \, f \bigl)$ of
bilinear maps
$$
\rightharpoonup \, \, : P \times V \to V, \quad \triangleleft : V
\times P \to V, \quad \vartheta \,\, : P \times P \to V, \quad
\triangleright \, : P \times V \to V, \quad f: P \times P \to V,
$$
satisfying the following compatibilities for all $p$, $q$, $r\in
P$ and $x\in V$:
\begin{enumerate}
\item[(Ab1)] $(V, \, \rightharpoonup, \, \triangleleft, \,
\triangleright) \in {}_P^P\Mm_P$ is a Poisson bimodule over $P$

\item[(Ab2)] $\vartheta(p, \, q) \lhd r = \vartheta(p, \, qr) -
\vartheta(pq, \, r) + p \rightharpoonup \vartheta(q, \, r)$

\item[(Ab3)] $f(p, p) = 0$

\item[(Ab4)] $f(p, [q, \, r ] ) + f(q, [r, \, p ]) + f(r, [p, \,
q] ) + p\triangleright f(q, r) + q\triangleright f(r, p) + r
\triangleright f(p, q) = 0$

\item[(Ab5)] $f (pq, \, r) - f(p,\, r) \triangleleft q - \, p
\rightharpoonup f(q, \, r) = r \triangleright \vartheta(p, \, q) +
\vartheta ([p, \, r], \, q) + \vartheta (p, \, [q, \, r])$.
\end{enumerate}
These are the axioms that remain from (H1)-(H7), (L1)-(L4) and
(P1)-(P7) in the case that $\cdot_V = [-, \, -]_V = 0$. In this
context \deref{echiaa} takes the following simplified form: two
local crossed systems $\bigl(\rightharpoonup, \, \triangleleft, \,
\vartheta, \, \triangleright, \, f \bigl)$ and
$\bigl(\rightharpoonup', \, \triangleleft', \, \vartheta', \,
\triangleright', \, f' \bigl)$ of $P$ by $V_0$ are local
cohomologous and we denote this by $\bigl(\rightharpoonup, \,
\triangleleft, \, \vartheta, \, \triangleright, \, f \bigl) \,
\approx_{l, \, 0} \, \bigl(\rightharpoonup', \, \triangleleft', \,
\vartheta', \, \triangleright', \, f' \bigl)$ if and only if
$\rightharpoonup = \rightharpoonup'$, $\triangleleft =
\triangleleft'$, $\triangleright = \triangleright'$ and there
exists a linear map $r: P \to V$ such that
\begin{eqnarray}
\vartheta(p, \, q) &=& \vartheta '(p, \, q) + p \rightharpoonup
r(q) + r(p) \lhd  q - r(p q)  \eqlabel{cocab1}\\
f (p, \, q) &=& f' (p, \, q) + p \triangleright r(q) - q
\triangleright r(p) - r \bigl([p, \, q]\bigl) \eqlabel{cocab2}
\end{eqnarray}
for all $p$, $q\in P$. The equalities $\rightharpoonup \, = \,
\rightharpoonup'$, $\triangleleft = \triangleleft'$,
$\triangleright = \triangleright'$ show that the object $
{\mathcal P} {\mathcal H}^{2} \, (P, \, V_0)$ is also a coproduct
in the category of sets over all triples $(\rightharpoonup,
\triangleleft, \triangleright)$ such that $(V, \, \rightharpoonup,
\, \triangleleft, \, \triangleright) \in {}_P^P\Mm_P$ is a Poisson
bimodule. We will highlight this fact by fixing the Poisson module
structure $(V, \rightharpoonup, \triangleleft, \triangleright) \in
{}_P^P\Mm_P$ on $V$ and denoting by ${\mathcal
C}_{(\rightharpoonup, \triangleleft, \triangleright)} (P, \, V)$
the set of all pairs $(\vartheta, f)$ consisting of two bilinear
maps $\vartheta \, : P \times P \to V$ and $f: P \times P \to V$
satisfying the compatibility conditions (Ab2) - (Ab5). Two pairs
$(\vartheta, f)$ and $(\vartheta', f') \in {\mathcal
C}_{(\rightharpoonup, \triangleleft, \triangleright)} (P, \, V)$
are local cohomologous if there exists a linear map $r: P \to V$
satisfying the compatibility conditions
\equref{cocab1}-\equref{cocab2}. If we denote by ${\mathcal P}
{\mathcal H}^{2}_{(\rightharpoonup, \triangleleft,
\triangleright)} \, (P, \, V)$ the quotient set of ${\mathcal
C}_{(\rightharpoonup, \triangleleft, \triangleright)} (P, \, V)$
via this equivalence relation we obtain the following:

\begin{corollary}\colabel{cazuabspargere}
Let $P$ be a Poisson algebra and $V$ a vector space with the
abelian Poisson algebra structure $V_0$. Then:
\begin{equation}\eqlabel{balsoi2}
{\mathcal P} {\mathcal H}^{2} \, (P, \, V_0) \, = \,
\sqcup_{(\rightharpoonup, \triangleleft, \triangleright)} \,
{\mathcal P} {\mathcal H}^{2}_{(\rightharpoonup, \triangleleft,
\triangleright)} \, (P, \, V)
\end{equation}
where the coproduct in the right hand side is in the category of
sets over all possible Poisson bimodule structures
$(\rightharpoonup, \, \triangleleft, \, \triangleright)$ on $V$.
\end{corollary}

Now we shall take a step forward and we will consider $P$ as well
with the abelian Poisson algebra structure, i.e. $P = P_0$. Having
in mind the theory of groups\footnote{We recall that a group is
called metabelian if it is an extension of an abelian group by an
abelian group.} we introduce the following:

\begin{definition}\delabel{cotangent}
A Poisson algebra $Q$ is called \emph{metabelian} if $Q$ is an
extension of an abelian Poisson algebra by another abelian Poisson
algebra.
\end{definition}

Using \prref{hocechiv} and then \thref{1} we obtain that any
metabelian Poisson algebra $Q$ is isomorphic to a crossed product
$P_0 \, \sharp \, V_0$, for some vector spaces $P = P_0$ and $V =
V_0$, i.e. $P_0 \, \sharp \, V_0$ has the multiplication and the
bracket given for any $p$, $q\in P$, $x$, $y \in V$ by:
\begin{eqnarray}
(p, x) \star (q, y) &:=& \bigl(0, \,\, \vartheta (p, q) + p
\rightharpoonup y + x \triangleleft q \bigl)
\eqlabel{cotg1} \\
\{ (p, x), \, (q, y) \} &:=& \bigl( 0 , \,\, f(p, q) + p
\triangleright y - q \triangleright x  \bigl) \eqlabel{cotg2}
\end{eqnarray}
for some bilinear maps
$$
\rightharpoonup \, \, : P \times V \to V, \quad \triangleleft : V
\times P \to V, \quad \vartheta \,\, : P \times P \to V, \quad
\triangleright \, : P \times V \to V, \quad f: P \times P \to V
$$
satisfying the following compatibility conditions for any $p$,
$q$, $r\in P$ and $x\in V$:
\begin{eqnarray}
&& (p \rightharpoonup x) \lhd q = p \rightharpoonup (x \lhd q),
\quad \vartheta(p, \, q) \lhd r = p \rightharpoonup \vartheta(q,
\, r) \eqlabel{cotang1} \\
&& p \rightharpoonup (q \rightharpoonup x) = (x \lhd p) \lhd q = 0 \eqlabel{cotang2} \\
&& f (p, p) = 0, \quad p \triangleright (q \triangleright x) =
q\triangleright (p \triangleright x) \eqlabel{cotang3} \\
&&  p\triangleright f(q, r) + q\triangleright f(r, p) + r
\triangleright f(p, q) = 0 \eqlabel{cotang4} \\
&& r \triangleright \vartheta(p, \, q) + f(p,\, r) \triangleleft q
+ \, p \rightharpoonup f(q, \, r) = 0 \eqlabel{cotang5} \\
&& p \rightharpoonup (q \triangleright x) = - \, (p \triangleright
x) \triangleleft q = q \triangleright (p \rightharpoonup x) \eqlabel{cotang6} \\
&& (q \triangleright x) \triangleleft p = q \triangleright (x
\triangleleft p) \eqlabel{cotang7}
\end{eqnarray}
which are the ones remaining from the axioms (P0)-(P7) in this
context. For two vector spaces $P$ and $V$ we denote by ${\mathcal
M} {\mathcal A} \, (P, \, V)$ the set of all \emph{metabelian
systems of $P$ by $V$}, that is all bilinear maps
$(\rightharpoonup, \, \triangleleft, \, \vartheta, \,
\triangleright, \, f)$ satisfying the compatibility conditions
\equref{cotang1}-\equref{cotang7}. With these notations we obtain
the following result that classifies all metabelian Poisson
algebras that are extensions of $P_0$ by $V_0$:

\begin{corollary}\colabel{Poissmeta}
Let $P$ and $V$ be two vector spaces with the abelian Poisson
algebra structure. Then:
\begin{equation}\eqlabel{balsoi5}
{\mathcal P} {\mathcal H}^{2} \, (P_0, \, V_0) \cong {\mathcal M}
{\mathcal A} \, (P, \, V)/\approx^{a}_{m}
\end{equation}
where $\approx^{a}_{m}$ is the following relation on ${\mathcal M}
{\mathcal A} \, (P, \, V)$: $(\rightharpoonup, \, \triangleleft,
\, \vartheta, \, \triangleright, \, f) \approx^{a}_{m}
(\rightharpoonup', \, \triangleleft', \, \vartheta', \,
\triangleright', \, f')$ if and only if $\rightharpoonup =
\rightharpoonup'$, $\triangleleft = \triangleleft'$,
$\triangleright = \triangleright'$ and there exists a linear map
$r: P \to V$ such that for any $p$, $q\in P$ we have:
\begin{eqnarray}
\vartheta(p, \, q) = \vartheta '(p, \, q) + p \rightharpoonup r(q)
+ r(p) \lhd  q, \quad f (p, \, q) = f' (p, \, q) + p
\triangleright r(q) - q \triangleright r(p). \eqlabel{metne1}
\end{eqnarray}
\end{corollary}

\begin{proof}
Follows from the above considerations once we observe that the
equivalence relation from \deref{echiaa} takes the form given in
the statement for abelian Poisson algebras $P_0$ and $V_0$.
\end{proof}

\begin{remark}
The way we defined the equivalence relation $\approx^{a}_{m}$ in
\coref{Poissmeta} indicates the decomposition of ${\mathcal P}
{\mathcal H}^{2} \, (P_0, \, V_0)$ into a coproduct over all
triples $(\rightharpoonup, \, \triangleleft, \, \triangleright)$
such that $(V, \, \rightharpoonup, \, \triangleleft, \,
\triangleright)$ is a Poisson bimodule over the abelian Poisson
algebra $P_0$. The decomposition is just the special case of the
one given in \coref{cazuabspargere} by taking $P = P_0$.
\end{remark}

The classification of all metabelian Poisson algebras of a given
dimension seems to be a very challenging problem which will be
addressed elsewhere. However we also include here two examples:
the first one classifies all crossed products $k_0 \# k_0$ by
computing the object ${\mathcal P} {\mathcal H}^{2} \, (k_0, \,
k_0)$; in particular, we will classify all $2$-dimensional
metabelian Poisson algebras.

\begin{example}\exlabel{cot2dim}
Let $P = V := k_0$ be the abelian Poisson algebra of dimension
$1$. Then ${\mathcal P} {\mathcal H}^{2} \, (k_0, \, k_0) \cong \{
(a, \, b) \in k^2 \, | \, ab = 0 \}$. The explicit bijection
between $\{ (a, \, b) \in k^2 \, | \, ab = 0 \}$ and the set of
all equivalence classes of all non-cohomologous extensions of
$k_0$ by $k_0$ is given by: $(a, \, b) \mapsto k^{2}_{a, \, b} $,
where $k^{2}_{a, \, b} $ is the Poisson algebra with the basis
$\{e_1, \, e_2\}$ and the structure: $e_1 \star e_1 := a \, e_2$
and $[e_1, e_2] := b\, e_2$, for all $(a, \, b) \in \{ (a, \, b)
\in k^2 \, | \, ab = 0 \}$.

In particular, up to an isomorphism of Poisson algebras, there
exist three $2$-dimensional metabelian Poisson algebras: $k^2_0$,
$k^2_{(1, 0)}$ and $k^2_{(0, 1)}$.

Indeed, in the first step we can easily see that there exists a
bijection ${\mathcal M} {\mathcal A} \, (k_0, \, k_0) \cong \{ (a,
b) \in k^2 \, | \, ab = 0 \}$ given such that the metabelian
system $(\rightharpoonup, \, \triangleleft, \, \vartheta, \,
\triangleright, \, f)$ associated to the pair $(a, b)$ with $ab =
0$ is given by:
$$
\rightharpoonup \, := 0, \,\,\, \triangleleft := 0, \,\,\, f:= 0,
\,\,\, \vartheta (1, 1) := a, \,\,\, 1\triangleright 1 := b.
$$
Then, we observe that the equivalence relation of
\coref{Poissmeta} written equivalently on $\{ (a, \, b) \in k^2 \,
| \, ab = 0 \}$ becomes $(a, \, b) \, \approx^{a}_{m} \, (a', \,
b')$ if and only if $a = a'$ and $b = b'$. Thus, ${\mathcal P}
{\mathcal H}^{2} \, (k_0, \, k_0) = {\mathcal M} {\mathcal A} \,
(k_0, \, k_0)$. The Poisson algebra $k^{2}_{a, \, b}$ is just the
crossed product $k_0 \# k_0$ with the structures given by
\equref{cotg1}-\equref{cotg2} associated to above metabelian
system. We have considered the canonical basis in $k\times k$.
\end{example}

Let $n$ be a positive integer. We describe all crossed products
$k_0 \# \, k^n_0$ (these are $(n+1)$-dimensional metabelian
Poisson algebras) and then we shall classify all extension of
$k_0$ by $k^n_0$ by computing ${\mathcal P} {\mathcal H}^{2} \,
(k_0, \, k^n_0)$. The tool for both questions is an interesting
set of matrices defined as follows: Let $\mathfrak{C} (n)$ be the
set of all $4$-uples $(A, \, B, \, C, \, \theta_0) \in {\rm M}_n
(k) \times {\rm M}_n (k) \times {\rm M}_n (k) \times k^n$
satisfying the following compatibilities:
\begin{equation}\eqlabel{meta13}
AB = BA, \,\, A C = CA = - BC = - CB, \,\, A^2 = B^2 = 0, \,\, A
\theta_0 = B \theta_0, \,\, C \theta_0 = 0.
\end{equation}

Two $4$-uples $(A, \, B, \, C, \, \theta_0)$ and $(A', \, B', \,
C', \, \theta'_0)$ are cohomologous and we denote this by $(A, \,
B, \, C, \, \theta_0) \approx^{a}_{m} (A', \, B', \, C', \,
\theta'_0)$ if and only if $A = A'$, $B = B'$, $C = C'$ and there
exists $r \in k^n$ such that $\theta_0 - \theta'_0 = (A + B) r $.
With these notations we have:

\begin{example}\exlabel{cot3dima}
Any crossed product $k_0 \# \, k^n_0$ is isomorphic to the Poisson
algebra denoted by $k^{n+1}_{(A, \, B, \, C, \, \theta_0)}$ which
is the vector space with basis $\{E_1, \, E_2, \cdots,  E_{n+1}
\}$, multiplication $\star$ and bracket given by:
\begin{eqnarray}
E_i \star E_{n+1} &:=& \sum_{j=1}^n \, b_{ji} \, E_j, \quad E_{n+1}
\star E_i := \sum_{j=1}^n \, a_{ji} \, E_j \\
E_{n+1} \star E_{n+1} &:=& \sum_{j=1}^n \, \theta_{0j} \, E_j, \quad
\{E_{n+1}, \, E_i  \} := \sum_{j=1}^n \, c_{ji} \, E_j
\end{eqnarray}
for all $i = 1, \cdots, n$ and $(A, \, B, \, C, \, \theta_0) \in
\mathfrak{C} (n)$ - all undefined operations are zero and $A =
(a_{ij})$, $B = (b_{ij})$, $C = (c_{ij})$ and $ \theta_0 =
(\theta_{0j})$. Furthermore,
$$
{\mathcal P} {\mathcal H}^{2} \, (k_0, \, k^n_0) \cong
\mathfrak{C} (n)/ \approx^{a}_{m}.
$$

Indeed, by a straightforward computation we can show that there
exists a bijection between the set ${\mathcal M} {\mathcal A} \,
(k_0, \, k^n_0)$ of all metabelian systems of $k_0$ by $k^n_0$ and
the set of all $4$-tuples $(\lambda, \Lambda, \gamma, \theta_0)$
consisting of three linear maps $\lambda$, $\Lambda$, $\gamma :
k^n \to k^n$ and a vector $\theta_0 \in k^n$ satisfying the
following compatibility conditions:
\begin{eqnarray}
&& \Lambda \circ \lambda = \lambda \circ \Lambda, \quad \Lambda
(\theta_0) = \lambda (\theta_0), \quad \lambda^2 = \Lambda^2 = 0
\eqlabel{compos1} \\
&&\gamma (\theta_0) = 0, \quad \lambda \circ \gamma = \gamma \circ
\lambda = - \Lambda \circ \gamma = - \gamma \circ \Lambda.
\eqlabel{compos2}
\end{eqnarray}
The bijection is given such that the metabelian system
$(\rightharpoonup, \, \triangleleft, \, \vartheta, \,
\triangleright, \, f)$ corresponding to $(\lambda, \Lambda,
\gamma, \theta_0)$ is given by:
$$
a \rightharpoonup x := a \, \lambda (x), \,\,\, x \triangleleft a
:= a \, \Lambda (x), \,\,\, \vartheta (a, \, b):= ab \, \theta_0,
\,\,\, a \triangleright x := a \, \gamma (x), \,\,\, f(a, b):= 0
$$
for all $a$, $b\in k$ and $x\in k^n$. We denote by $\{e_1, \cdots,
e_n \}$ the canonical basis of $k^n$ and let $A$, $B$, $C$ be the
matrices associated to $\lambda$, $\Lambda$ and respectively
$\gamma$ with respect to this basis. If we take as a basis in $k
\times k^n$ the vectors $E_1 = (0, e_1), \cdots$, $E_n = (0, e_n)$
and $E_{n+1} = (1, 0)$ then the Poisson algebra $k^{n+1}_{(A, \,
B, \, C, \, \theta_0)}$ is just the crossed product $k_0 \# k^n_0$
with the structures given by \equref{cotg1}-\equref{cotg2}
associated to the above metabelian system written equivalently
with matrices. Finally, the equivalence relation of
\coref{Poissmeta} rephrased using the elements of $\mathfrak{C}
(n)$ is precisely the one given before \exref{cot3dima}.
\end{example}

Similar computations to the ones performed in \exref{cot3dima}
lead to the description and classification of all crossed products
$k^n_0 \, \# \, k_0$:

\begin{example}\exlabel{cot3dimb}
Any crossed product $k^n_0 \, \# \, k_0$ is isomorphic to one of
the two families of Poisson algebras $k^{n+1}_{(\theta, \, f)}$ or
$k^{n+1}_{(\gamma, \, f)}$ described below.

\begin{eqnarray}
&k^{n+1}_{(\theta, \, f)}:& \qquad
E_i \star E_j : = \theta (e_i, \, e_j) \, E_{n+1} \qquad
\{E_i, \, E_j \} := f(e_i, \, e_j) \, E_{n+1} \eqlabel{adouaneb1}
\end{eqnarray}
for all bilinear maps $\theta: k^n \times k^n \to k$, $f: k^n \times k^n \to k$
satisfying the condition $f (x, x) = 0$, for all $x\in k^n$. We denote by ${\mathcal M} {\mathcal A}_1 \, (k^n, \, k)$
the set of all such pairs $(\theta, \, f)$.

\begin{eqnarray}
&k^{n+1}_{(\gamma, \, f)}:& \qquad
\{E_i, \, E_j \} := f(e_i, \, e_j) \, E_{n+1}, \qquad \{E_i, \, E_{n+1} \} :=  \gamma (e_i) \, E_{n+1}
\eqlabel{adouaneb2}
\end{eqnarray}
for all pairs $(\gamma, \, f)$ consisting of a non-trivial linear map
$\gamma: k^n \to k$ and a bilinear map $f: k^n \times k^n \to k$
satisfying the following two compatibility conditions:
$$
f(x, x) = 0, \qquad \gamma (x) \, f(y, \, z) + \gamma (y) \, f(z, \, x) + \gamma (z) \, f (x, \, y) = 0
$$
for all $x$, $y$, $z\in k^n$. We denote by ${\mathcal M} {\mathcal A}_2 \, (k^n, \, k)$
the set of all such pairs $(\gamma, \, f)$. Furthermore, we have that
$$
{\mathcal P} {\mathcal H}^{2} \, (k^n_0, \, k_0) \cong {\mathcal
M} {\mathcal A}_1 \, (k^n, \, k) \, \sqcup \, {\mathcal M}
{\mathcal A}_2 \, (k^n, \, k).
$$

The details are left to the reader. We just mention that the set
${\mathcal M} {\mathcal A} \, (k^n, \, k)$ of all metabelian
systems of $k^n_0$ by $k_0$ is parameterized by the set of all
triples $(\theta, \, \gamma, \, f)$ consisting of two bilinear
maps $\theta: k^n \times k^n \to k$, $f: k^n \times k^n \to k$ and
a linear map $\gamma: k^n \to k$ satisfying the following
compatibility conditions:
\begin{equation}\eqlabel{splitaremeta}
f(x, x) = 0, \quad \gamma (x) \, f(y, \, z) + \gamma (y) \, f(z, \, x) + \gamma (z) \, f (x, \, y) = 0, \quad
\theta (x, \, y) \gamma (z) = 0
\end{equation}
for all $x$, $y$, $z\in k^n$. The bijection is given such that the metabelian system
$(\rightharpoonup, \, \triangleleft, \, \vartheta, \,
\triangleright, \, f)$ corresponding to $(\theta, \, \gamma, \, f)$ is given by:
$$
\rightharpoonup  := 0, \,\,\,\, \triangleleft
:= 0,  \,\,\,\, \vartheta (x, \, y):=  \theta (x, \, y),
\,\,\, x \triangleright a := a \, \gamma (x), \,\,\, f(x, y):= f(x, \, y)
$$
for all $a\in k$ and $x$, $y\in k^n$. The last equation of
\equref{splitaremeta} leads us to consider two cases (depending on
whether $\lambda$ is the trivial map or not) and to write the set
${\mathcal M} {\mathcal A} \, (k^n, \, k)$ as a coproduct of
${\mathcal M} {\mathcal A}_1 \, (k^n, \, k)$ and ${\mathcal M}
{\mathcal A}_2 \, (k^n, \, k)$. The Poisson algebras
$k^{n+1}_{(\theta, \, f)}$ and $k^{n+1}_{(\gamma, \, f)}$ are just
the corresponding crossed products $k^n_0 \, \# \, k_0$. The
equivalence relation written on each of the sets ${\mathcal M}
{\mathcal A}_i \, (k^n, \, k)$ becomes an equality and this proves
our last assertion.
\end{example}

The description of all $3$-dimensional metabelian Poisson algebras
follows just by taking $n = 2$ in \exref{cot3dima} and
\exref{cot3dimb}:

\begin{corollary}
Any $3$-dimensional metabelian Poisson algebra is isomorphic to
one of the following Poisson algebras: $k^{3}_{(A, \, B, \, C, \,
\theta_0)}$, for all $(A, \, B, \, C, \, \theta_0) \in
\mathfrak{C} (2)$,  or $k^{3}_{(\theta, \, f)}$, for all $(\theta,
\, f) \in {\mathcal M} {\mathcal A}_1 \, (k^2, \, k)$, or
$k^{3}_{(\gamma, \, f)}$, for all $(\gamma, \, f)\in {\mathcal M}
{\mathcal A}_2 \, (k^2, \, k)$.
\end{corollary}

\section{Co-flag Poisson Algebras} \selabel{coflag}

\coref{formulamare} shows that computing the global cohomological
object ${\mathcal G} {\mathcal P} {\mathcal H}^{2} \, (P, \, V)$
is a highly non-trivial task. In this section we deal with a
special class of Poisson algebras for which we provide a recursive
method  to compute this classification object. These are the
so-called co-flag Poisson algebras as defined below:

\begin{definition} \delabel{coflg}
Let $P$ be a Poisson algebra and $E$ a vector space. A Poisson
algebra structure $(\cdot_E, [- , \, -]_E)$ on $E$ is called a
\emph{co-flag Poisson algebra over $P$} if there exists a positive
integer $n$ and a finite chain of epimorphisms of Poisson algebras
\begin{equation} \eqlabel{lant}
P_n : = (E, \cdot_E, [- , \, -]_E)
\stackrel{\pi_{n}}{\longrightarrow} P_{n-1}
\stackrel{\pi_{n-1}}{\longrightarrow} P_{n-2} \, \cdots \,
\stackrel{\pi_{2}}{\longrightarrow} P_1 \stackrel{\pi_{1}}
{\longrightarrow} P_{0} := P {\longrightarrow} 0
\end{equation}
such that ${\rm dim}_k ( {\rm Ker} (\pi_{i}) ) = 1$, for all $i =
1, \cdots, n$. A finite dimensional Poisson algebra is called a
\emph{co-flag Poisson algebra} if it is a co-flag Poisson algebra
over $\{0\}$.
\end{definition}

The recursive method we introduce in this section relies on the
first step, namely $n=1$. Therefore we will start by describing
and classifying all co-flag Poisson algebra structures over $P$ of
dimension $1 + {\rm dim}_k (P)$ or, equivalently, all crossed
products between $P$ and a vector space $V_{1}$ of dimension $1$.
This process can be iterated by replacing the initial Poisson
algebra $P$ by such a crossed product between $P$ and $V_{1}$.
Before proving the main theoretical result which allows us to
develop this iteration process we need to introduce a few pieces
of terminology:

\begin{definition} \delabel{coflagab}
Let $P$ be a Poisson algebra. An \emph{abelian co-flag datum of
$P$} is a $5$-tuple $(\lambda, \, \Lambda, \, \theta, \, \gamma,
\, f)$, where $\lambda$, $\Lambda$, $\gamma: P \to k$ are linear
maps, $\theta$, $f : P\times P \to k$ are bilinear maps satisfying
the following compatibilities for any $p$, $q$, $r\in P$:
\begin{enumerate}
\item[(AF1)] $\lambda$ and $\Lambda: P \to k$ are morphisms of
associative algebras

\item[(AF2)] $\theta (p, \, qr) - \theta (pq, \, r) = \theta (p,
\, q) \Lambda(r) - \theta (q, \, r) \lambda(p)$

\item[(AF3)] $f (p, p) = \gamma ([p, \, q]) = 0$

\item[(AF4)] $f(p, [q, \, r ] ) + f(q, [r, \, p ]) + f(r, [p, \,
q] ) + \gamma (p) f(q, r) + \gamma (q) f(r, p) + \gamma (r) f(p,
q) = 0$

\item[(AF5)] $f(pq, \, r) - \Lambda(q) f(p, \, r) - \lambda(p)
f(q, \,r) = \gamma (r) \theta (p, \, q) + \theta ([p, \, r], \, q)
+ \theta (p, \, [q, \, r])$

\item[(AF6)] $\gamma (pq) = \gamma (p) \Lambda(q) + \lambda(p)
\gamma(q) $

\item[(AF7)] $ \lambda( [p, \, q]) = \Lambda ([p, \, q]) = 0.$
\end{enumerate}
We denote by ${\mathcal A} {\mathcal F} \, (P)$ the set of all
abelian co-flag data of $P$.
\end{definition}

\begin{remark}
\deref{coflagab} simplifies considerably in the case of perfect
Poisson algebras. Indeed, let $P$ be a perfect Poisson algebra,
i.e. $P$ is generated as a vector space by all brackets $[x,y]$
with $x$, $y \in P$. Then, it follows from the compatibilities
(AF3) and (AF7) that $\lambda = \Lambda = \gamma \equiv 0$. Thus,
an abelian co-flag datum reduces to a pair of bilinear maps
$(\theta, \, f)$ satisfying the following compatibilities for all
$p$, $q$, $r \in P$:
\begin{eqnarray*}
&&f(p,  p) = 0, \quad \theta(p, qr) = \theta(pq, r)\\
&&f\bigl(p, [q, r]\bigl) + f\bigl(q, [r, p]\bigl) + f\bigl(r, [p, q]\bigl) = 0\\
&&f(pq, r) = \theta([p, r], q) + \theta(p, [q, r]).
\end{eqnarray*}
\end{remark}

\begin{definition} \delabel{coflagneab}
Let $P$ be a Poisson algebra. A \emph{non-abelian co-flag datum of
$P$} is a triple $(\lambda, \, \theta, \, u)$, where $u \in k
\setminus \{0\}$, $\lambda: P \to k$ is a linear map, $\theta :
P\times P \to k$ is a bilinear map satisfying the following two
compatibilities for any $p$, $q$, $r\in P$:
\begin{enumerate}
\item[(NF1)] $\lambda (pq) = \lambda(p) \lambda(q) - u \, \theta
(p, \, q)$

\item[(NF2)] $\theta (p, \, qr) - \theta (pq, \, r) = \theta (p,
\, q) \lambda(r) - \theta (q, \, r) \lambda(p).$
\end{enumerate}
\end{definition}

We denote by ${\mathcal N} {\mathcal F} \, (P)$ the set of all
non-abelian co-flag data of $P$ and by ${\mathcal F} \, (P) :=
{\mathcal A} {\mathcal F} \, (P) \,  \sqcup \, {\mathcal N}
{\mathcal F} \, (P)$ the disjoint union of the two sets. The
elements of ${\mathcal F} \, (P)$ will be called \emph{co-flag
data} of $P$. In light of the following result co-flag data are
the key tools for describing co-flag Poisson algebras:

\begin{proposition}\prlabel{coflagdim1}
Let $P$ be a Poisson algebra and $V$ a vector space of dimension
$1$ with a basis $\{x\}$. Then there exists a bijection between
the set ${\mathcal P} {\mathcal S} \, (P, \, V)$ of all crossed
systems of $P$ by $V$ and the set ${\mathcal F} \, (P) = {\mathcal
A} {\mathcal F} \, (P) \, \sqcup \, {\mathcal N} {\mathcal F} \,
(P)$ of all co-flag data of $P$.

The bijection is such that the crossed system $\Omega(P, V) =
\bigl(\rightharpoonup, \, \triangleleft, \, \vartheta, \, \cdot_V,
\, \triangleright, \, f, [-, -]_V \bigl)$ corresponding to
$(\lambda, \, \Lambda, \, \theta, \, \gamma, \, f) \in {\mathcal
A} {\mathcal F} \, (P)$ is given by:
\begin{eqnarray}
p \rightharpoonup x &=& \lambda (p) x, \quad \,\,\,\,\,\,\,\,
x \triangleleft p = \Lambda(p) x, \quad \vartheta (p, q) = \theta (p, q) x \eqlabel{extenddim1.1a} \\
x \cdot_V x  &=& \left[x, \, x\right]_V = 0, \qquad p
\triangleright x = \gamma (p) x, \qquad \,\,\,\,\, f(p, q) = f(p,
q) x \eqlabel{extenddim1.1b}
\end{eqnarray}
while the crossed system $\Omega(P, V)  = \bigl(\rightharpoonup,
\, \triangleleft, \, \vartheta, \, \cdot_V, \, \triangleright \,
f, [-, -]_V \bigl)$ corresponding to $(\lambda, \, \theta, \, u)
\in {\mathcal N} {\mathcal F} \, (P) $ is given by:
\begin{eqnarray}
p \rightharpoonup x &=& \lambda (p) x, \quad \,\,\,\,\,\,\,\,
x \triangleleft p = \lambda(p) x, \quad \vartheta (p, q) = \theta (p, q) x \eqlabel{extenddim1.2a}\\
x \cdot_V x  &=& u \, x , \quad \,\, p \triangleright x = 0, \quad
\,\,\, f(p, q) = - u^{-1} \, \lambda ( [p, \, q] ) \, x, \quad
\left[x, \, x\right]_V = 0 \eqlabel{extenddim1.2b}
\end{eqnarray}
for all $p$, $q \in P$.
\end{proposition}

\begin{proof}
We have to compute all crossed systems $\Omega(P, V)  =
\bigl(\rightharpoonup, \, \triangleleft, \, \vartheta, \, \cdot_V,
\, \triangleright \, f, [-, -]_V \bigl)$ between $P$ and $V$, i.e.
all bilinear maps which satisfy the compatibilities (H0) -- (H7),
(L0) -- (L4) and (P0) -- (P7). Since $V$ has dimension $1$ and
$(V, [-, \, -]_V)$ is a Lie algebra it follows that $ [-, \, -]_V
= 0$, the trivial map. Once again by the fact that $V$ has
dimension $1$ we obtain that any crossed system from $\Omega(P,
V)$ is uniquely determined by three linear maps $\lambda$,
$\Lambda$, $\gamma: P \to k$, two bilinear maps $\theta$, $f :
P\times P \to k$ and a scalar $u \in k$ via the following
formulas:
\begin{eqnarray*}
p \rightharpoonup x &=& \lambda (p) x, \quad \,\,\,\,\,\,\,\,
x \triangleleft p = \Lambda(p) x, \quad \vartheta (p, q) = \theta (p, q) x \\
x \cdot_V x  &=& u \, x, \qquad p \triangleright x = \gamma (p) x,
\qquad \,\,\,\,\, f(p, q) = f(p, q) x
\end{eqnarray*}
for all $p\in P$. We are left to derive the axioms that need to be
fulfilled by the above maps in order for compatibilities (H0) --
(H7), (L0) -- (L4) and (P0) -- (P7) to hold.

We start by writing down the compatibility (H2) which implies $u
\Lambda (p) = u \lambda (p)$, for all $p \in P$. Therefore we
distinguish two cases, namely $u = 0$ or $u \neq 0$. Suppose first
that $u = 0$. We will prove that in this case the axioms (H0) --
(H7), (L0) -- (L4) and (P0) -- (P7) hold if and only if $(\lambda,
\, \Lambda, \, \theta, \, \gamma, \, f)$ is an abelian co-flag
datum. It is easy to see that the compatibilities (H0), (H1), (H3)
and (H4) are trivially fulfilled. (H6) and (H7) are equivalent to
the fact that $\lambda$ and $\Lambda$ are algebra maps and thus
(AF1) holds while (H5) collapses to (AF2). Furthermore, (L0) and
(L2) are trivially fulfilled, (L1) together with (L3) collapses to
(AF3) while (L4) is equivalent to (AF4). (P1) collapses to (AF5),
(P2) is equivalent to (AF6) while (P4), (P6) and (P7) are
trivially fulfilled. Finally, (AF7) is derived from (P3) and (P5).

Suppose now that $u \neq 0$ and thus $\Lambda = \lambda$. In this
case (P3) comes down to $f(p, q) = - u^{-1} \, \lambda ( [p, \, q]
)$ while (P4) implies that $\gamma = 0$. It can be easily seen
that (H5) collapses to (NF2) while (H6) is equivalent to (NF1).
Based on these facts it is straightforward to see that the other
compatibilities are trivially fulfilled. We will only check (P1).
Indeed, the left hand side of (P1) gives:
\begin{eqnarray*}
LHS (P1) &=& f(pq,\,r) - \lambda(q) f(p,\,r) - \lambda(p)
f(q,\,r)\\
&=& \bigl(- \lambda(\underline{[pq,\,r]}) + \lambda(q) \lambda([p,
\, r]) + \lambda(p) \lambda([q,\, r])\bigl)u^{-1}\\
&\stackrel{\equref{p1}}{=}& \bigl(-\underline{\lambda([p,\, r] q)}
-\underline{\lambda(p [q, \, r])} + \lambda(q) \lambda([p, \, r])
+ \lambda(p) \lambda([q,\, r])\bigl)u^{-1}\\
&\stackrel{(NF1)}{=}& \Bigl(- \lambda([p, \, r]) \lambda(q) -
\lambda(p) \lambda([q, \, r]) + \theta([p, \, r], \, q) u +
\theta(p, \, [q, \, r]) u +\\
&&\lambda(q) \lambda([p, \, r]) + \lambda(p) \lambda([q,\,
r])\Bigl)u^{-1}\\
&=& \theta([p, \, r], \, q) + \theta(p, \, [q, \, r]) = RHS (P1)
\end{eqnarray*}
and the proof is now finished.
\end{proof}

Let $(\lambda, \, \Lambda, \, \theta, \, \gamma, \, f) \in
{\mathcal A} {\mathcal F} \, (P)$. The crossed product $P \,
\#_{(\lambda, \, \Lambda, \, \theta, \, \gamma, \, f)} \, V$
associated to the crossed system given by \equref{extenddim1.1a} -
\equref{extenddim1.1b} will be denoted by $P_{{\mathcal A}}( x
\,|\, (\lambda, \, \Lambda, \, \theta, \, \gamma, \, f))$ and has
the Poisson algebra structure defined by:
\begin{eqnarray*}
(p,\, 0) \star (q,\, 0) &=& (pq, \, \theta(p, q) x), \quad \,
\{(p,0), \, (q, 0)\} = ([p, q], \, f(p, q) x) \\
(p, \, 0) \star (0, \, x) &=& (0, \, \lambda(p) x), \quad \, \, \,
\,\,\,\,\,\, \{(p, 0), \, (0, x)\} = (0, \, \gamma(p) x)\\
(0, \, x) \star (p, \, 0) &=& (0, \, \Lambda(p) x), \quad \, \, \,
\,\,\,\,\,\, \{(0, x), \, (p, 0)\} = (0, \, - \gamma(p) x).
\end{eqnarray*}
On the other hand, for $(\lambda, \, \theta, \, u) \in {\mathcal
N} {\mathcal F} \, (P) $, the crossed product $P \, \#_{(\lambda,
\, \theta, \, u)} \, V$ associated to the crossed system given by
\equref{extenddim1.2a} - \equref{extenddim1.2b} will be denoted by
$P_{{\mathcal N}}( x \, | \, (\lambda, \, \theta, \, u))$ and has
the Poisson algebra structure defined by:
\begin{eqnarray*}
(p,\, 0) \star (q,\, 0) &=& (pq, \, \theta(p, q) x), \quad \,
(p,0)\star (0, x) = (0, \, \lambda(p) x) \\
(0, \, x) \star (0, \, x) &=& (0, \, u x), \quad
\,\,\,\,\,\,\,\,\,\,\,\,\,\,\,\,\, (0, x) \star (p, 0) = (0, \, \lambda(p) x)\\
\{(p, 0), \, (q, 0)\} &=& \bigl([p, q], \, -u^{-1} \lambda([p, q])
x\bigl).
\end{eqnarray*}

The first explicit classification result for the GE-problem
follows: it is also the key step in the classification of all
co-flag Poisson algebras algebras over $P$.

\begin{theorem}\thlabel{clsfP1}
Let $P$ be a Poisson algebra. Then:
$$
{\mathcal G} {\mathcal P} {\mathcal H}^{2} \, (P, \, k) \, \cong
\, ({\mathcal A} {\mathcal F} \, (P) /\equiv_1) \sqcup ({\mathcal
N} {\mathcal F} \, (P)/\equiv_2)\qquad {\rm where:}
$$
$\equiv_1$ is the equivalence relation on the set ${\mathcal A}
{\mathcal F} \, (P)$ defined as follows: $(\lambda, \, \Lambda, \,
\theta, \, \gamma, \, f) \equiv_1 (\lambda ', \, \Lambda ', \,
\theta ', \, \gamma ', \, f ')$ if and only of $\lambda = \lambda
'$, $\Lambda = \Lambda '$, $\gamma = \gamma '$ and there exists a
linear map $r:P \to k$ such that for any $p$, $q \in P$:
\begin{eqnarray}
\theta(p, \,q) &=& \theta ' (p, \, q) + r(q) \, \lambda'(p) +
r(p) \,\Lambda'(q) - r(pq)\\
f(p, \, q) &=& f'(p, \, q) + r(q)\, \gamma^{'}(p) - r(p) \,
\gamma^{'}(q) - r([p,\,q]),
\end{eqnarray}
$\equiv_2$ is the equivalence relation on ${\mathcal N} {\mathcal
F} \, (P)$ given by: $(\lambda, \, \theta, \, u) \equiv_2 (\lambda
', \, \theta ', \, u ')$ if and only if $u = u '$ and there exists
a linear map $r: P \to k$ such that for any $p$, $q \in P$:
\begin{eqnarray}
\lambda(p) &=& \lambda ' (p) + r(p) \, u\\
\theta(p,\,q) &=& \theta '(p,\,q) + r(q)\, \lambda'(p) + r(p) \,
\lambda'(q) - r(pq) + r(p)\,r(q) \,u.
\end{eqnarray}
\end{theorem}

\begin{proof}
By applying \prref{coflagdim1} for $V = k$ we know that the set
${\mathcal P} {\mathcal S} \, (P, k)$ of all crossed systems of
$P$ by $k$ is in bijection to the set ${\mathcal F} \, (P) =
{\mathcal A} {\mathcal F} \, (P) \, \sqcup \, {\mathcal N}
{\mathcal F} \, (P)$ of all co-flag data of $P$. Thus, the problem
reduces to computing the set $\Bigl({\mathcal A} {\mathcal F} \,
(P) \, \sqcup \, {\mathcal N} {\mathcal F} \, (P)\Bigl) /
\approx$. Based on these facts, a little computation shows that
the compatibility conditions from \deref{echiaa}, imposed for the
crossed systems \equref{extenddim1.1a}-\equref{extenddim1.1b} and
respectively \equref{extenddim1.2a}-\equref{extenddim1.2b}, take
precisely the form given in the statement of the theorem. To this
end we should notice that an abelian co-flag datum is never
cohomologous to a non-abelian co-flag datum thanks to the
compatibility condition (M1) from \leref{HHH}.
\end{proof}

We provide a first example which relies on \thref{clsfP1}:

\begin{example} \exlabel{calexpext}
Let $\mathfrak{H}$ be the Heisenberg Lie algebra with the basis
$\{h_{1}, h_{2}, h_{3}\}$ and the bracket defined by $ [h_{1}, \,
h_{2}] = h_{3}$. $\mathfrak{H}$ admits a Poisson algebra structure
\cite{gozeremm} with the associative multiplication given by
$h_{1}^{2} = h_{3}$. Then ${\mathcal G} {\mathcal P} {\mathcal
H}^{2} \, (\mathfrak{H}, \, k) \cong k^{4} \, \sqcup \, k^{*} \,
\sqcup \, k^{*} \, \sqcup \, (k^* \times k^*) \, \sqcup \, k^{*}$
and the equivalence classes of all non-cohomologous
$4$-dimensional co-flag algebras over $\mathfrak{H}$ are
represented by the following Poisson algebras with basis $\{e_1,
\, e_2, \, e_3, \, e_4 \}$, multiplication and bracket given by:
\begin{eqnarray*}
&\mathfrak{H}^{\tau}_{\lambda, \zeta, \xi}:& \quad e_{1} \star
e_{1} = e_{3}, \, e_{2} \star e_{2} = \lambda e_{4}, \, e_{3}
\star e_{3} = \zeta e_{4}, \, e_{4} \star e_{4} = \xi e_{4},\\
&& \quad \{e_{1}, \, e_{2}\} = e_{3} + \tau e_{4}, \,
{\rm where} \, \tau, \lambda, \zeta, \xi \in k\\
&\mathfrak{H}^{v}:& \quad e_{1} \star e_{1} = e_{3}, \, \{e_{1},\, e_{2}\}
= e_{3}, \, \{e_{2}, \, e_{4}\} = v e_{4}, \,  {\rm where} \, v \in k^{*}\\
&\overline{\mathfrak{H}}^{v}:& \quad e_{1} \star e_{1} = e_{3}, \, \{e_{1},\, e_{2}\}
= e_{3}, \, \{e_{1}, \, e_{4}\} = v e_{4}, \,  {\rm where} \, v \in k^{*}\\
&\mathfrak{H}^{v, w}:& \quad e_{1} \star e_{1} = e_{3}, \, \{e_{1},\, e_{2}\}
= e_{3}, \, \{e_{1}, \, e_{4}\} = v e_{4}, \, \{e_{2}, \, e_{4}\} = w e_{4},  \,
{\rm where} \, v, w \in k^{*}\\
&\widetilde{\mathfrak{H}}^{v}:& \quad e_{1} \star e_{1} = e_{3} + v e_{4}, \,
\{e_{1}, \, e_{2}\} = e_{3},  \,  {\rm where} \, v \in k^{*}.\\
\end{eqnarray*}
To this end we use first \prref{coflagdim1}. However, as the
computations are rather long but straightforward we will only
indicate the main steps of the proof. It can be shown that the
abelian co-flag data of $\mathfrak{H}$ are given as follows:
\begin{eqnarray*}
&(\lambda_{1}, \, \Lambda_{1}, \, \theta_{1}, \, \gamma_{1},\, f_{1}):&
\quad \lambda_{1} = \Lambda_{1}= \gamma_{1} \equiv 0, \, f_{1}(e_{1}, e_{2}) = - f_{1}(e_{2}, e_{1}) = \sigma\\
&& \quad \theta_{1}(e_{1}, e_{1}) = \beta_{1}, \, \theta_{1}(e_{2}, e_{2})
= \beta_{2}, \, \theta_{1}(e_{1}, e_{2}) = \beta_{3}, \, \theta_{1}(e_{2}, e_{1}) = \beta_{4}\\
&& \quad {\rm with}\,\, \sigma, \beta_{i} \in k, i \in \{1, 2, 3\}\\
&(\lambda_{2}, \, \Lambda_{2}, \, \theta_{2}, \, \gamma_{2},\,
f_{2}):&
\quad \lambda_{2} = \Lambda_{2} \equiv 0, \, \gamma_{2}(e_{2})= v, \,\theta_{2}(e_{1}, e_{1}) = \omega,\\
&& \quad f_{2}(e_{1}, e_{2}) = - f_{2}(e_{2}, e_{1}) =
\eta, \, f_{2}(e_{3}, e_{2}) = -f_{2}(e_{2}, e_{3}) = v \omega,\\
&& \quad {\rm with}\,\, v \in k^{*},\, \omega, \eta \in k\\
&(\lambda_{3}, \, \Lambda_{3}, \, \theta_{3}, \, \gamma_{3},\, f_{3}):&
\quad \lambda_{3} = \Lambda_{3} \equiv 0, \, \gamma_{3}(e_{1}) = w, \, \theta_{3}(e_{1}, e_{1}) = \zeta,\\
&& \quad f_{3}(e_{1}, e_{2}) = - f_{3}(e_{2}, e_{1})
= \mu, \, f_{3}(e_{3}, e_{1}) = - f_{3}(e_{1}, e_{3}) = - w \zeta,\\
&& \quad {\rm with}\,\, w \in k^{*},\, \mu, \zeta \in k\\
&(\lambda_{4}, \, \Lambda_{4}, \, \theta_{4}, \, \gamma_{4},\, f_{4}):&
\quad \lambda_{4} = \Lambda_{4} \equiv 0, \, \gamma_{4}(e_{1})
= \overline{v}, \, \gamma_{4}(e_{2}) = \overline{w}, \, \theta_{4}(e_{1}, e_{1}) = \nu,\\
&& \quad f_{4}(e_{1}, e_{2}) = -f_{4}(e_{2}, e_{1})
= \delta, \, f_{4}(e_{3}, e_{1}) = - f_{4}(e_{1}, e_{3}) = \overline{v} \nu,\\
&& \quad f_{4}(e_{3}, e_{2}) = - f_{4}(e_{2}, e_{3}) =
\overline{w}\nu, \,\, {\rm with}\,\, \overline{v}, \overline{w}
\in k^{*}, \, \nu, \delta \in k.
\end{eqnarray*}
By \thref{clsfP1} it follows that two abelian co-flag data
$(\lambda_{i}, \, \Lambda_{i}, \, \theta_{i}, \, \gamma_{i},\,
f_{i})$ and $(\lambda_{j}, \, \Lambda_{j}, \\ \, \theta_{j}, \,
\gamma_{j},\, f_{j})$ with $i \neq j$ are never equivalent since
$\gamma_{i}$ is obviously different from $\gamma_{j}$. Using again
\thref{clsfP1} we can easily notice that an abelian co-flag datum
$(\lambda_{1}, \, \Lambda_{1}, \, \theta_{1}, \, \gamma_{1},\,
f_{1})$ implemented by $\sigma$, $\beta_{i} \in k$, $i \in
\{1,..., 4\}$ is equivalent to an abelian co-flag datum
$(\lambda_{1}, \, \Lambda_{1}, \, \theta_{1}, \, \gamma_{1},\,
f_{1})$ implemented by $\sigma - \beta_{1}$, $0$, $\beta_{2}$,
$\beta_{3}$, $\beta_{4}$. The Poisson algebras denoted by
$\mathfrak{H}^{\tau}_{\lambda, \zeta, \xi}$ are obtained from
these type of co-flag data. Moreover, it is straightforward to see
that an abelian co-flag datum $(\lambda_{2}, \, \Lambda_{2}, \,
\theta_{2}, \, \gamma_{2},\, f_{2})$ implemented by $v \in
k^{*}$,\, $\omega$, $\eta \in k$ is equivalent to an abelian
co-flag datum $(\lambda_{2}, \, \Lambda_{2}, \, \theta_{2}, \,
\gamma_{2},\, f_{2})$ implemented by $v \in k^{*}$ and $\omega =
\eta = 0$ which gives rise to the Poisson algebra denoted by
$\mathfrak{H}^{v}$. In the same manner any abelian co-flag datum
$(\lambda_{3}, \, \Lambda_{3}, \, \theta_{3}, \, \gamma_{3},\,
f_{3})$ implemented by $w \in k^{*}$,\, $\mu$, $\zeta \in k$ is
equivalent to an abelian co-flag datum $(\lambda_{3}, \,
\Lambda_{3}, \, \theta_{3}, \, \gamma_{3},\, f_{3})$ implemented
by $w \in k^{*}$ and $\mu = \zeta = 0$ while any abelian co-flag
datum $(\lambda_{4}, \, \Lambda_{4}, \, \theta_{4}, \,
\gamma_{4},\, f_{4})$ implemented by $\overline{v}$, $\overline{w}
\in k^{*}$, \, $\nu$, $\delta \in k$ is equivalent to an abelian
co-flag datum $(\lambda_{4}, \, \Lambda_{4}, \, \theta_{4}, \,
\gamma_{4},\, f_{4})$ implemented by $\overline{v}$, $\overline{w}
\in k^{*}$ and $\nu = \delta = 0$. The latter two abelian co-flag
data give rise to the Poisson algebras denoted by
$\overline{\mathfrak{H}}^{v}$ and $\mathfrak{H}^{v, w}$
respectively.

On the other hand, the non-abelian co-flag data of $\mathfrak{H}$
are given as follows:
\begin{eqnarray*}
&(\lambda, \theta, u):& \quad \lambda(h_{i}) = \alpha_{i} \in k, \,\, i \in \{1, 2, 3\}, \quad u \in k^{*}, \\
&& \quad
\begin{tabular}{c|ccc}
  $\theta$ & $h_1$ & $h_2$ & $h_3$ \\\hline
  $h_1$   & $(\alpha_{1}^{2} - \alpha_{3})u^{-1}$ & $\alpha_{1} \alpha_{2} u^{-1}$ & $\alpha_{1} \alpha_{3} u^{-1}$   \\
  $h_2$   & $\alpha_{1} \alpha_{2} u^{-1}$ &  $\alpha_{2}^{2}u^{-1}$  & $\alpha_{2} \alpha_{3} u^{-1}$    \\
  $h_3$   & $\alpha_{3} \alpha_{1} u^{-1}$ & $\alpha_{3} \alpha_{2} u^{-1}$  & $\alpha_{3}^{2} u^{-1}$   \\
\end{tabular}
\end{eqnarray*}
By a routine computation based on \thref{clsfP1} we obtain that
the non-abelian co-flag datum $(\lambda, \theta, u)$ implemented
by $u \in k^{*}$, \, $\alpha_{1}$, $\alpha_{2}$, $\alpha_{3}$ is
equivalent to the non-abelian co-flag datum $(\lambda, \theta, u)$
implemented by $u \in k^{*}$, $0$, $0$, $0$ which gives rise to
the Poisson algebra denoted by $\widetilde{\mathfrak{H}}^{u}$.
Thus ${\mathcal G} {\mathcal P} {\mathcal H}^{2} \, (\mathfrak{H},
\, k) \cong k^{4} \, \sqcup \, k^{*} \, \sqcup \, k^{*} \, \sqcup
\, k^{*2} \, \sqcup \, k^{*}$.
\end{example}

Next we will highlight the efficiency of \thref{clsfP1} in
classifying co-flag Poisson algebras of small dimension. We start
by computing all $2$-dimensional co-flag Poisson algebras. Since
there are two non-isomorphic Poisson algebra structures on a
vector space of dimension one, namely $k_{0}$ and $k_{1}$, we have
to consider the cases $P = k_{0}$ and $P = k_{1}$.

\begin{corollary}\colabel{dim2Poissoncoflag}
Let $k$ be a field. Then ${\mathcal G} {\mathcal P} {\mathcal
H}^{2} \, (k_{0}, \, k) \cong k^{*}\, \sqcup \, k \, \sqcup \,
k^{*}$ and $ {\mathcal G} {\mathcal P} {\mathcal H}^{2} \, (k_{1},
\, k) \cong \{*\} \, \sqcup \, \{*\} \, \sqcup k \, \sqcup \, k \,
\sqcup \, k^{*} $ where $\{*\}$ denotes the singleton set.
Explicitly, any $2$-dimensional co-flag Poisson algebra is
cohomologous to one of the following Poisson algebras with basis
$\{e_{1}, \, e_{2}\}$, multiplication and bracket defined by:
\begin{eqnarray*}
& k_{0, \delta}^{2}: & \quad e_{1} \star e_{1} = \delta e_{2}, \,
{\rm where} \,\delta \in k^{*}\\
& {}_{\mu}k_{0}^{2}: & \quad \{e_{1}, \, e_{2}\} = \mu e_{2},\,
{\rm where} \, \mu \in k\\
& k_{0,u}^{2}:& \quad e_{2} \star e_{2} = u e_{2}, \, {\rm where} \, u \in k^{*}\\
& k_{1}^{2}: & \quad e_{1} \star e_{1} = e_{1}\\
& \overline{k}_{1}^{2}: &\quad e_{1} \star e_{1} = e_{1}, \, e_{1}
\star e_{2} = e_{2} \star e_{1} = e_{2}
\end{eqnarray*}
\begin{eqnarray*}
& {}_{\mu}k_{1}^{2}: & \quad e_{1} \star e_{1} =
 e_{1}, \, e_{1} \star e_{2} =  e_{2}, \, \{e_{1}, \, e_{2}\} = \mu e_{2},\,
{\rm where} \, \mu \in k\\
& {}_{\mu}\overline{k}_{1}^{2}: & \quad e_{1} \star e_{1} =
 e_{1}, \, e_{2} \star e_{1} =  e_{2}, \, \{e_{1}, \, e_{2}\} = \mu e_{2},\,
{\rm where} \, \mu \in k\\
& k_{u}^{2}: & \quad e_{1} \star e_{1} = e_{1}, \, e_{2} \star
e_{2} = u e_{2},  \, {\rm where} \, u \in k^{*}.
\end{eqnarray*}
\end{corollary}

\begin{proof}
In what follows we consider $\{y\}$ as a basis in $P = k$. Suppose
first that $P = k_{0}$. Then the abelian co-flag data of $k_0$ are
given as follows:
\begin{eqnarray*}
&(\lambda_{1}, \, \Lambda_{1}, \, \theta_{1}, \, \gamma_{1},\, f_{1}):&
\quad \lambda_{1} = \Lambda_{1}= \gamma_{1} \equiv 0, \,\,
\theta_{1}(y, y) = \delta, \,\, f_{1}\equiv 0, \, {\rm with}\,\, \delta \in k\\
&(\lambda_{2}, \, \Lambda_{2}, \, \theta_{2}, \, \gamma_{2},\,
f_{2}):& \quad \lambda_{2} = \Lambda_{2} \equiv 0, \,\, \theta_{2}
= f_{2} \equiv 0, \,\, \gamma_{2}(y) = \mu, \, {\rm with}\,\, \mu
\in k^{*}.
\end{eqnarray*}
An abelian co-flag datum $(\lambda_{1}, \, \Lambda_{1}, \,
\theta_{1},\, \gamma_{1},\, f_{1})$ implemented by $\delta \in k$
is equivalent to another abelian co-flag datum $(\lambda_{1}, \,
\Lambda_{1}, \, \theta_{1}, \, \gamma_{1}, \, f_{1})$ implemented
by $\delta ' \in k$ if and only if $\delta = \delta'$. In the same
manner it follows that an abelian co-flag datum $(\lambda_{2}, \,
\Lambda_{2}, \, \theta_{2}, \, \gamma_{2},\, f_{2})$ implemented
by $\mu \in k^{*}$ is equivalent to another abelian co-flag datum
$(\lambda_{2}, \, \Lambda_{2}, \, \theta_{2}, \, \gamma_{2},\,
f_{2})$ implemented by $\mu ' \in k^{*}$ if and only if $\mu = \mu
'$. Moreover, as the abelian co-flag data $(\lambda_{2}, \,
\Lambda_{2}, \, \theta_{2}, \, \gamma_{2},\, f_{2})$ are
implemented by a non-zero scalar $\mu \in k^{*}$ it follows that
$(\lambda_{2}, \, \Lambda_{2}, \, \theta_{2}, \, \gamma_{2},\,
f_{2})$ is never equivalent to a co-flag datum $(\lambda_{1}, \,
\Lambda_{1}, \, \theta_{1}, \, \gamma_{1},\, f_{1})$. The two
non-equivalent abelian co-flag data give rise to the Poisson
algebras denoted by $k_{0, \delta}^{2}$ and respectively
${}_{\mu}k_{0}^{2}$, $\delta \in k$, $\mu \in k^{*}$.

On the other hand, the non-abelian co-flag data of $k_0$ are given
as follows:
\begin{eqnarray*}
&(\lambda, \, \theta, \, u):& \quad \lambda(y) = \alpha, \,\,
\theta(y, y) = \alpha^{2} u^{-1}, \,\, {\rm with}\,\, \alpha \in
k,\, u \in k^{*}.
\end{eqnarray*}
Again by \thref{clsfP1} it follows that two non-abelian co-flag
data corresponding to scalars $(\alpha, u) \in k \times k^{*}$ and
respectively $(\alpha ', u') \in k \times k^{*}$ are equivalent if
and only if $u = u'$. Therefore any non-abelian co-flag datum
implemented by $(\alpha, u)$ is equivalent to a non-abelian
co-flag datum implemented by $(0, u)$ which gives rise to the
Poisson algebra denoted by $k_{0,u}^{2}$. Thus we obtain
${\mathcal G} {\mathcal P} {\mathcal H}^{2} \, (k_{0}, \, k) \cong
k^{*}\, \sqcup \, k \, \sqcup \, k^{*}$.

Now we turn to the second case, namely $P = k_{1}$. The abelian
co-flag data of $k_1$ are given as follows:
\begin{eqnarray*}
&(\lambda_{1}, \, \Lambda_{1}, \, \theta_{1}, \, \gamma_{1},\, f_{1}):&
\quad \lambda_{1} = \Lambda_{1}= \gamma_{1} \equiv 0, \,\, \theta_{1}(y, y)
= \delta, \,\, f_{1}\equiv 0, \, {\rm with}\,\, \delta \in k\\
&(\lambda_{2}, \, \Lambda_{2}, \, \theta_{2}, \, \gamma_{2},\,
f_{2}):& \quad \lambda_{2}(y) = \Lambda_{2}(y) = 1, \,\,
\theta_{2}(y, y) = \zeta, \,\, \gamma_{2} \equiv 0, \,\, f_{2}
\equiv 0, \, {\rm with}\,\, \zeta \in k\\
&(\lambda_{3}, \, \Lambda_{3}, \, \theta_{3}, \, \gamma_{3},\,
f_{3}):& \quad \lambda_{3}(y) = 1, \, \, \Lambda_{3} \equiv 0,
\,\,\gamma_{3}(y) =  \mu,\,\,
 \theta_{3} = f_{3} \equiv 0, \,  {\rm with}\,\, \mu \in k\\
&(\lambda_{4}, \, \Lambda_{4}, \, \theta_{4}, \, \gamma_{4},\,
f_{4}):& \quad \lambda_{4} \equiv 0, \, \, \Lambda_{4}(y) = 1,
\,\,\gamma_{4}(y) =  \mu,\,\,
 \theta_{4} = f_{4} \equiv 0, \,  {\rm with}\,\, \mu \in k.
\end{eqnarray*}
Since by \thref{clsfP1} two equivalent abelian co-flag data need
to have the same three maps $\lambda$, $\Lambda$ and $\gamma$ it
easily follows that an abelian co-flag datum $(\lambda_{i}, \,
\Lambda_{i}, \, \theta_{i}, \, \gamma_{i},\, f_{i})$ is never
equivalent to $(\lambda_{j}, \, \Lambda_{j}, \, \theta_{j}, \,
\gamma_{j},\, f_{j})$ if $i \neq j$. Moreover, for any $\delta \in
k$ the abelian co-flag datum $(\lambda_{1}, \, \Lambda_{1}, \,
\theta_{1}, \, \gamma_{1},\, f_{1})$ implemented by $\delta$ is
equivalent to the the abelian co-flag datum $(\lambda_{1}, \,
\Lambda_{1}, \, \theta_{1}, \, \gamma_{1},\, f_{1})$ implemented
by $0$ which give rise to the Poisson algebra denoted by
$k_{1}^{2}$. In the same manner for any $\zeta \in k$ the abelian
co-fag datum $(\lambda_{2}, \, \Lambda_{2}, \, \theta_{2}, \,
\gamma_{2},\, f_{2})$ implemented by $\zeta$ is equivalent to the
the abelian co-fag datum $(\lambda_{2}, \, \Lambda_{2}, \,
\theta_{2}, \, \gamma_{2},\, f_{2})$ implemented by $0$ which
gives rise to the Poisson algebra denoted by
$\overline{k}_{1}^{2}$. The situation changes if we look at the
abelian co-flag data $(\lambda_{3}, \, \Lambda_{3}, \, \theta_{3},
\, \gamma_{3},\, f_{3})$: two such data implemented by $\mu$ and
respectively $\mu '$ are equivalent if and only if $\mu = \mu'$.
The abelian co-flag datum $(\lambda_{3}, \, \Lambda_{3}, \,
\theta_{3}, \, \gamma_{3},\, f_{3})$ implemented by $\mu \in k$
gives rise to the Poisson algebra denoted by ${}_{\mu}k_{1}^{2}$.
Finally, the abelian co-flag datum $(\lambda_{4}, \, \Lambda_{4},
\, \theta_{4}, \, \gamma_{4},\, f_{4})$ gives rise to the Poisson
algebra denoted by ${}_{\mu}\overline{k}_{1}^{2}$.

The non-abelian co-flag data of $k_1$ are given as follows:
\begin{eqnarray*}
&(\lambda, \, \theta, \, u):& \quad \lambda(y) = \alpha, \,\,
\theta(y, y) = (\alpha^{2} - \alpha) u^{-1}, \,\, {\rm with}\,\,
\alpha \in k,\, u \in k^{*}.
\end{eqnarray*}
Two non-abelian co-flag data corresponding to scalars $(\alpha, u)
\in k \times k^{*}$ and respectively $(\alpha ', u') \in k \times
k^{*}$ are equivalent if and only if $u = u'$. Therefore any
non-abelian co-flag datum implemented by $(\alpha, u)$ is
equivalent to a non-abelian co-flag datum implemented by $(0, u)$
which gives rise to the Poisson algebra denoted by $k_{u}^{2}$.
\end{proof}

We can now apply our recursive method to the Poisson algebras
described in \coref{dim2Poissoncoflag}: we will obtain the
description and classification of all $3$-dimensional co-flag
Poisson algebras. In what follows we will only list the
$3$-dimensional co-flag Poisson algebras over ${}_{1}k^2_1$.

\begin{corollary}\colabel{dim2Poissoncoflag}
Any $3$-dimensional co-flag Poisson algebra over ${}_{1}k_{1}^{2}$
is cohomologous to one of the following Poisson algebras with
basis $\{e_{1}, \, e_{2}, \, e_{3}\}$:
\begin{eqnarray*}
& {}_{1}k_{1}^{3}: & \quad e_{1} \star e_{1} =  e_{1}, \, e_{1}
\star e_{2} = e_{2}, \, \{e_{1},
\, e_{2}\} = e_{2}\\
& {}_{1}\overline{k}_{1}^{3}: & \quad e_{1} \star e_{1} =  e_{1},
\, e_{1} \star e_{2} = e_{2}, \, e_{1} \star e_{3} = e_{3} \star
e_{1} = e_{3},\, \{e_{1}, \, e_{2}\} = e_{2}\\
& {}_{\nu}k_{1}^{3}:& \quad e_{1} \star e_{1} =  e_{1}, \, e_{1}
\star e_{2} = e_{2}, \, e_{1} \star e_{3} = e_{3}, \, \{e_{1}, \,
e_{3}\} = e_{3},\\
&& \quad \{e_{1},
\, e_{2}\} = e_{2} + \nu e_{3}, \,\, {\rm where} \,\, \nu \in k\\
& {}_{\omega}\widetilde{k}_{1}^{3}:& \quad e_{1} \star e_{1} =
e_{1}, \, e_{1} \star e_{2} = e_{2}, \, e_{1} \star e_{3} = e_{3},
\, \{e_{1}, \, e_{3}\} = \omega e_{3}, \,\, \{e_{1},
\, e_{2}\} = e_{2},\\
&& \quad {\rm where} \,\, \omega \in k - \{1\}\\
& {}_{\tau}\overline{k}_{1}^{3}:& \quad e_{1} \star e_{1} = e_{1},
\, e_{1} \star e_{2} = e_{2}, \, e_{3} \star e_{1} = e_{3}, \,
\{e_{1}, \, e_{3}\} = \tau e_{3}, \,\, \{e_{1},
\, e_{2}\} = e_{2},\\
&& \quad {\rm where} \,\, \tau \in k\\
& k_{1, u}^{3}: & \quad e_{1} \star e_{1} = e_{1}, \, e_{3} \star
e_{3} = u e_{3}, \, \{e_{1}, \, e_{2}\} = e_{2}, \,\, {\rm
where}\,\, u \in k^{*}.
\end{eqnarray*}
Furthermore, we have ${\mathcal G} {\mathcal P} {\mathcal H}^{2}
\, ({}_{1}k_{1}^{2}, \, k) \cong \{*\}\, \sqcup \{*\}\, \sqcup  \,
k \, \sqcup \, (k -\{1\}) \, \sqcup \, k \, \sqcup \, k^{*}$.
\end{corollary}

\begin{proof}
We denote by $\{y_{1}, y_{2}\}$ a $k$-basis of ${}_{1}k_{1}^{2}$.
The abelian co-flag data of ${}_{1}k_{1}^{2}$ are given as
follows:
\begin{eqnarray*}
&(\lambda_{1}, \, \Lambda_{1}, \, \theta_{1}, \, \gamma_{1},\,
f_{1}):& \quad \lambda_{1} = \Lambda_{1}= \gamma_{1} \equiv 0,
\,\, \theta_{1}(y_{1}, y_{1}) = \alpha, \,\, \theta_{1}(y_{1},
y_{2}) = \beta,\\
&& \quad f_{1}(y_{1}, y_{2}) = - f_{1}(y_{2},
y_{1}) = \beta, \, {\rm with}\,\, \alpha, \beta \in k\\
&(\lambda_{2}, \, \Lambda_{2}, \, \theta_{2}, \, \gamma_{2},\,
f_{2}):& \quad \lambda_{2}(y_{1}) = \Lambda_{2}(y_{1})= 1, \,
\gamma_{2} \equiv 0, \,\, \theta_{2}(y_{1}, y_{1}) = \zeta, \,\,
\theta_{2}(y_{2}, y_{1}) = \delta,\\
&& \quad f_{2}(y_{2}, y_{1}) = - f_{2}(y_{1}, y_{2}) = \delta, \,
{\rm with}\,\, \zeta, \delta \in k\\
&(\lambda_{3}, \,
\Lambda_{3}, \, \theta_{3}, \, \gamma_{3},\, f_{3}):& \quad
\lambda_{3}(y_{1}) = 1, \, \Lambda_{3} \equiv 0,
\, \theta_{3} \equiv 0, \,\, \gamma_{3}(y_{1}) = \omega,\\
&& \quad f_{3}(y_{1}, y_{2}) = - f_{3}(y_{2},
y_{1}) = \nu, \, {\rm with}\,\, \omega, \nu \in k\\
&(\lambda_{4}, \, \Lambda_{4}, \, \theta_{4}, \, \gamma_{4},\,
f_{4}):& \quad \lambda_{4} \equiv 0, \,  \Lambda_{4} (y_{1}) = 1,
\, \theta_{4} \equiv 0, \,\, \gamma_{4}(y_{1}) = \tau, \, f_{4}
\equiv 0 \, {\rm with}\,\, \tau \in k.
\end{eqnarray*}
To start with we should notice that in the light of \thref{clsfP1}
two equivalent abelian co-flag data need to have the same three
maps $\lambda$, $\Lambda$ and $\gamma$ and thus an abelian
co-\newpage flag datum $(\lambda_{i}, \, \Lambda_{i}, \,
\theta_{i}, \, \gamma_{i},\, f_{i})$ is never equivalent to
$(\lambda_{j}, \, \Lambda_{j}, \, \theta_{j}, \, \gamma_{j},\,
f_{j})$ if $i \neq j$. Using again \thref{clsfP1} we obtain that
an abelian co-flag datum $(\lambda_{1}, \, \Lambda_{1}, \,
\theta_{1}, \, \gamma_{1},\, f_{1})$ implemented by $(\alpha,\,
\beta)$ is equivalent to the abelian co-flag datum $(\lambda_{1},
\, \Lambda_{1}, \, \theta_{1}, \, \gamma_{1},\, f_{1})$
implemented by $(0, \, 0)$. The latter gives rise to the Poisson
algebra ${}_{1}k_{1}^{3}$. A co-flag datum $(\lambda_{2}, \,
\Lambda_{2}, \, \theta_{2}, \, \gamma_{2},\, f_{2})$ implemented
by $(\zeta,\, \delta)$ is equivalent to the abelian co-flag datum
implemented by $(0, \, 0)$. This gives rise to the Poisson
algebras denoted by ${}_{1}\overline{k}_{1}^{3}$. Consider now two
co-flag data $(\lambda_{3}, \, \Lambda_{3}, \, \theta_{3}, \,
\gamma_{3},\, f_{3})$ implemented by $(\omega,\, \nu)$ and
respectively $(\omega',\, \nu')$. In order for the two co-flag
data to be equivalent we need to have $\omega = \omega'$. If
$\omega = \omega ' = 1$ then the co-flag data are equivalent if
and only if $\nu = \nu '$; we denote the corresponding Poisson
algebra by ${}_{\nu}k_{1}^{3}$. On the other hand, if $\omega =
\omega ' \neq 1$ then the two co-flag data are always equivalent
and therefore any such co-flag datum is equivalent to the co-flag
datum implemented by $(\omega, \, 0)$; the corresponding Poisson
algebra will be denoted by by ${}_{\omega}\widetilde{k_{1}}^{3}$.
Finally two co-flag data $(\lambda_{4}, \, \Lambda_{4}, \,
\theta_{4}, \, \gamma_{4},\, f_{4})$ corresponding to $\tau$ and
respectively $\tau '$ are equivalent if and only if $\tau = \tau
'$. The corresponding Poisson algebras are denoted by
${}_{\tau}\overline{k}_{1}^{3}$.

On the other hand, the non-abelian co-flag data of
${}_{1}k_{1}^{2}$ are given as follows:
\begin{eqnarray*}
&(\lambda, \theta, u):& \quad \lambda(y_{i}) = \alpha_{i} \in k,
\,\, i \in \{1, 2\}, \,\, u \in k^{*},\, \,
\begin{tabular}{c|cc}
  $\theta$ & $y_1$ & $y_2$  \\\hline
  $y_1$   & $(\alpha_{1}^{2} - \alpha_{1})u^{-1}$ & $(\alpha_{1} \alpha_{2} - \alpha_{2}) u^{-1}$   \\
  $y_2$   & $\alpha_{1} \alpha_{2} u^{-1}$ &  $\alpha_{2}^{2}u^{-1}$     \\
\end{tabular}
\end{eqnarray*}
A straightforward computation based on \thref{clsfP1} proves that
the non-abelian co-flag datum $(\lambda, \theta, u)$ implemented
by $u \in k^{*}$, $\alpha_{1}$, $\alpha_{2}$ is equivalent to the
non-abelian co-flag datum $(\lambda, \theta, u)$ implemented by $u
\in k^{*}$, $0$, $0$ which gives rise to the Poisson algebra
denoted by $k_{1, u}^{3}$. The proof is now finished.
\end{proof}

\end{document}